\documentclass[reqno,a4paper, 11pt]{amsart}

\usepackage[a4paper=true,pdfpagelabels]{hyperref}
\usepackage{graphicx}

\usepackage[ansinew]{inputenc}
\usepackage{amsfonts,epsfig}
\usepackage{latexsym}
\usepackage{amsmath}
\usepackage{amssymb}

\newtheorem{theorem}{Theorem}
\newtheorem{lemma}[theorem]{Lemma}

\newtheorem{proposition}[theorem]{Proposition}

\theoremstyle{definition}

\theoremstyle{remark}

\numberwithin{equation}{section}

\newcommand{\D}{\mathbb{D}}
\newcommand{\DD}{\widehat{\mathcal{D}}}
\newcommand{\N}{\mathbb{N}}
\newcommand{\R}{\mathbb{R}}

\newcommand{\C}{\mathbb{C}}

\renewcommand{\phi}{\varphi}

\newcommand{\T}{\mathbb{T}}

\def\a{\alpha}       \def\b{\beta}        \def\g{\gamma}
\def\d{\delta}       \def\De{{\Delta}}    
\def\la{\lambda}     \def\om{\omega}      
       \def\t{\theta}       
         \def\r{\rho}         \def\z{\zeta}
                  \def\vp{\varphi}
\def\G{\Gamma}

\def\R{{\mathcal R}}
\def\I{{\mathcal I}}

\begin{document}

\title[Embedding theorems for Bergman spaces via harmonic analysis]{Embedding theorems for Bergman spaces via harmonic analysis}


\thanks{This research was supported in part by the Ram\'on y Cajal program
of MICINN (Spain); by Ministerio de Edu\-ca\-ci\'on y Ciencia, Spain, projects
MTM2011-25502 and MTM2011-26538;  by   La Junta de Andaluc{\'i}a, (FQM210) and
(P09-FQM-4468);  by Academy of Finland project no. 268009,  by V\"ais\"al\"a Foundation of Finnish Academy of Science and Letters, and by Faculty of Science and Forestry of University of Eastern Finland project no. 930349.
}

\thanks{The authors wish the thank Brett Wick for pointing out the reference \cite{Sawyer} with regard to maximal functions.}

\author{Jos\'e \'Angel Pel\'aez}
\address{Departamento de An\'alisis Matem\'atico, Universidad de M\'alaga, Campus de
Teatinos, 29071 M\'alaga, Spain} \email{japelaez@uma.es}

\author{Jouni R\"atty\"a}
\address{University of Eastern Finland, P.O.Box 111, 80101 Joensuu, Finland}
\email{jouni.rattya@uef.fi}
\date{\today}


\begin{abstract}
Let $A^p_\om$ denote the Bergman space in the unit disc induced by a radial weight~$\omega$
with the doubling property $\int_{r}^1\om(s)\,ds\le C\int_{\frac{1+r}{2}}^1\om(s)\,ds$.
The positive Borel measures such that the differentiation operator of order $n\in\N\cup\{0\}$ is bounded from $A^p_\om$ into $L^q(\mu)$ are characterized in terms of geometric conditions when $0<p,q<\infty$. En route to the proof a theory of tent spaces for weighted Bergman spaces is built.
\end{abstract}

\maketitle



\section{Introduction}

The purpose of this paper is to characterize, in terms of geometric conditions,
the positive Borel measures $\mu$ on the unit disc $\D$ such that
    \begin{equation}\label{eq:main}
    \left(\int_{\D}|f^{(n)}(z)|^q\,d\mu(z)\right)^{\frac1q}\le C\|f\|_{A^p_\om},
    \end{equation}
where $n\in\N\cup\{0\}$ and $A^p_\om$ is the Bergman space induced by~$\omega$ in the class $\widehat{\mathcal{D}}$ which consists of the radial weights $\om$ such that $\widehat{\om}(r)=\int_r^1\om(s)\,ds$ satisfies ${\widehat{\om}(r)}/{\widehat{\om}(\frac{1+r}{2})}$ is bounded.
The inequality \eqref{eq:main} for the standard weighted Bergman spaces has been studied by many authors since seventies~\cite{Luecking1985,Lu93}. The problem has also been solved for the Bekoll\'e-Bonami and, if $n=0$, the rapidly decreasing weights~\cite{OC,PP}.

A weight $\om\in\DD$ cannot decrease exponentially but may grow such that the integral over the pseudohyperbolic disc $\Delta(a,r)$ with respect to $\om dA$ exceeds that of the Carleson square $S(a)$ in decay as $a$ approaches the boundary. In this case $\frac{\om(z)}{(1-|z|)^\eta}$ does not belong to the Bekoll\'e-Bonami class $B_p(\eta)$. In the half plane the measures satisfying \eqref{eq:main} have been characterized in~\cite{JacPartPott}, when $q=p\ge 1$,~$n=0$ and $\om$ satisfies a doubling property. In the unit disc, the case $q\ge p$ and $n=0$ has been considered in \cite[Chapter~2]{PelRat} for a subclass of $\DD$. The approach taken here, to be explained together with the statements, is different from those employed in the above-mentioned references and works for all $0<p,q<\infty$ and $n\in\N\cup\{0\}$.

For a positive Borel measure $\mu$ on~$\D$ and $\a>0$, we define the weighted maximal function
    $$
    M_{\om,\alpha}(\mu)(z)=\sup_{z\in S(a)}\frac{\mu(S(a))}{\left(\om\left(S(a)
    \right)\right)^\a},\quad
    z\in\D.
    $$
In the case $\a=1$ we simply write $M_{\om}(\mu)$, and if $\mu$ is of the form $\vp\om\,dA$, then $M_{\om,\a}(\mu)$ is the weighted maximal function  $M_{\om,\a}(\vp)$ of $\vp$. These maximal functions as well as the non-tangential approach regions
    \begin{equation*}\label{eq:gammadeuintro}
    \Gamma(\z)=\left\{z\in \D:\,|\t-\arg
    z|<\frac12\left(1-\frac{|z|}{r}\right)\right\},\quad
    \z=re^{i\theta}\in\D\setminus\{0\},
    \end{equation*}
and the related tents $T(z)=\left\{\z\in \D:\,z\in\Gamma(\z)\right\}$ will be frequently used in our study. Our first result characterizes the $q$-Carleson measures for $A^p_\om$.

\begin{theorem}\label{Theorem:CarlesonMeasures}
Let $0<p,q<\infty$, $\om\in\DD$ and $\mu$ be a positive Borel measure on~$\D$.
\begin{itemize}
\item[\rm(a)] If $p>q$, the following conditions are equivalent:
\begin{enumerate}
\item[\rm(i)] $\mu$ is a $q$-Carleson measure for $A^p_\om$;
\item[\rm(ii)] The function
    $$
    B_\mu(z)=\int_{\Gamma(z)}\frac{d\mu(\z)}{\om(T(\z))},\quad
    z\in\D\setminus\{0\},
    $$
belongs to $L^{\frac{p}{p-q}}_\om$;
\item[\rm(iii)] $M_\om(\mu)\in L^{\frac{p}{p-q}}_\om$.
\end{enumerate}
\item[\rm(b)] $\mu$ is a $p$-Carleson measure for $A^p_\omega$ if
and only if $M_{\om}(\mu)\in L^\infty$.
\item[\rm(c)] If $q>p$, the following conditions are equivalent:
\begin{enumerate}
\item[\rm(i)] $\mu$ is a $q$-Carleson measure for $A^p_\om$;
\item[\rm(ii)] $M_{\om,q/p}(\mu)\in L^\infty$;
\item[\rm(iii)]  $\displaystyle z\mapsto\frac{\mu\left(\Delta(z,r)\right)}{(\om\left(S(z)
    \right))^\frac{q}{p}}$ belongs to $L^\infty$ for any fixed $r\in(0,1)$.
\end{enumerate}
\end{itemize}
\end{theorem}

In the case $q\ge p$ we use a method from the theory of Hardy spaces to control the distribution function $\mu(\{z:N(f)(z)>\lambda\})$ of $N(f)(z)=\sup_{\z\in\G(z)}|f(\z)|$ by
    $$
    \int_{\{z:N(f)(z)>\lambda\}}\left(M_{\om,q/p}(\mu)(\z)\right)^\frac{p}{q}\om(\z)\,dA(\z),\quad 0<\lambda<\infty.
    $$
The existing literature~\cite{OC,Lu93,PP} suggests that for $p>q$ the most natural test functions are obtained by an atomic decomposition~\cite{Ro:de}. However, this approach does not seem to be adequate for the class $\DD$. As in the case $q\ge p$, methods from harmonic analysis turn out to be the appropriate ones. To some extent this is natural because a weighted Bergman space $A^p_\om$ induced by $\om\in\DD$ may lie essentially much closer to the Hardy space $H^p$ than any standard Bergman space $A^p_\a$~\cite{PelRat}. Luecking~\cite{Lu90} employed the theory of tent spaces, introduced by Coifman, Meyer and Stein~\cite{CMS} and further considered by Cohn and Verbitsky~\cite{CV}, to study the inequality~\eqref{eq:main} for Hardy spaces. In this setting one optimizes the use of the maximal and square area functions~\cite{BGS,FefSt}.
We will build an analogue of this theory for Bergman spaces and show that one can actually treat~\eqref{eq:main} for the Hardy and Bergman spaces simultaneously.

For $0<q<\infty$ and a positive Borel measure $\nu$ on $\D$, finite on compact sets, denote
    $
    A^q_{q,\nu}(f)(\z)=\int_{\Gamma(\z)}|f(z)|^q\,d\nu(z)
    $
and $A_{\infty,\nu}(f)(\z)=\nu\textrm{-ess}\sup_{z\in\Gamma(\z)}|f(z)|$. For $0<p<\infty$, $0<q\le\infty$ and $\om\in\DD$, the tent space $T^p_q(\nu,\om)$ consists of the $\nu$-equivalence classes of $\nu$-measurable functions $f$ such that $\|f\|_{T^p_q(\nu,\om)}=\|A_{q,\nu}(f)\|_{L^p_\om}$ is finite. For $0<q<\infty$, define
    $$
    C^q_{q,\nu}(f)(\z)=\sup_{a\in \Gamma(\z)}\frac{1}{\om(T(a))}\int_{T(a)}|f(z)|^q\om(T(z))\,d\nu(z).
    $$
A quasi-norm in the tent space $T^\infty_q(\nu,\om)$ is defined by $\|f\|_{T^\infty_q(\nu,\om)}=\|C_{q,\nu}(f)\|_{L^\infty}$.

The dual of $T^p_q(\nu,\om)$ is identified with $T^{p'}_{q'}(\nu,\om)$ when
$1\le p,q,<\infty$ and $p+q>2$. In the proof, a stopping time involving $A_{q,\nu}(f)$ and $C_{q,\nu}(f)$
is a fundamental step. The outcome of this reasoning is used to give an alternate proof of Theorem~\ref{Theorem:CarlesonMeasures}(b). It also plays a role in the proof of the factorization
    $$
    T^p_q(\nu,\om)=T^p_\infty(\nu,\om)\cdot T^\infty_q(\nu,\om),\quad 0<p,q<\infty,
    $$
which in turn is applied to obtain an atomic decomposition of functions in $T^p_q(\nu,\om)$ when $p\le q$. We will also show that $\left(T^p_q(\nu,\om)\right)^\star\simeq T^{p'}_\infty(\nu,\om)$ when $q<1<p$
and~$\nu$ is a discrete measure induced by a separated sequence. By pulling all these results together, we will achieve the statement of Theorem~\ref{Theorem:CarlesonMeasures}.

The Littlewood-Paley formula implies that for the classical Bergman spaces the question of when the differentiation operator $D^{(n)}(f)=f^{(n)}$ from $A^p_\a$ to $L^q(\mu)$ is bounded can be answered once the $q$-Carleson measures for $A^p_\a$ are characterized. However, this approach does not work for $\om\in\DD$ because an analogue of the Littlewood-Paley formula does not exist in this context~\cite[Proposition~4.3]{PelRat}.

\begin{theorem}\label{Theorem:DifferentiationOperator}
Let $0<p,q<\infty$, $\om\in\DD$, $n\in\N$ and $\mu$ be a positive Borel measure on~$\D$. Further, let $dh(z)=dA(z)/(1-|z|^2)^2$ denote the hyperbolic measure.
\begin{itemize}
\item[\rm(a)] If $p\ge q$, $D^{(n)}:A^p_\om\to L^q(\mu)$ is bounded if and only if, for any fixed $r\in(0,1)$, the function
    $$
    \Phi_\mu(z)=\frac{\mu(\Delta(z,r))}{\om(S(z))(1-|z|)^{qn}},\quad z\in\D,
    $$
belongs to
    \begin{enumerate}
    \item[\rm(i)] $T^\frac{p}{p-q}_{\frac{2}{2-q}}\left(h,\om\right)$, if $q<\min\{2,p\}$;
    \item[\rm(ii)] $T^\infty_{\frac{2}{2-p}}\left(h,\om\right)$, if $q=p<2$;
    \item[\rm(iii)] $T^{\frac{p}{p-q}}_\infty\left(h,\om\right)$, if $2\le q<p$.
    \end{enumerate}
\item[\rm(b)] If either $q>p$ or $2\le q=p$, the following conditions are equivalent:
\begin{enumerate}
\item[\rm(i)] $D^{(n)}:A^p_\om\to L^q(\mu)$ is bounded;
\item[\rm(ii)] $\displaystyle z\mapsto\frac{\mu\left(S(z)\right)}{(\om\left(S(z)
    \right))^\frac{q}{p}(1-|z|)^{nq}}$ belongs to $L^\infty$;
\item[\rm(iii)] $\displaystyle z\mapsto\frac{\mu\left(\Delta(z,r)\right)}{(\om\left(S(z)
    \right))^\frac{q}{p}(1-|z|)^{nq}}$ belongs to $L^\infty$ for any fixed $r\in(0,1)$.
\end{enumerate}
\end{itemize}
\end{theorem}

\section{Preliminary results}

For $0<p<\infty$ and a weight $\omega$, the weighted Bergman
space $A^p_\omega$ consists of analytic functions $f$ in the unit disc $\D$ for
which
    $$
    \|f\|_{A^p_\omega}^p=\int_\D|f(z)|^p\omega(z)\,dA(z)<\infty,
    $$
where $dA(z)=\frac{dx\,dy}{\pi}$ is the normalized
Lebesgue area measure on $\D$. We write
$A^p_\alpha$ for the classical weighted
Bergman space induced by
the standard weight $\omega(z)=(1-|z|^2)^\alpha$. A weight $\om$ is radial if $\omega(z)=\omega(|z|)$ for all $z\in\D$.

The Carleson square $S(I)$ based on an
interval $I\subset\T$ is the set $S(I)=\{re^{it}\in\D:\,e^{it}\in I,\,
1-|I|\le r<1\}$, where $|E|$ denotes the Lebesgue measure of $E\subset\T$. We associate to
each $a\in\D\setminus\{0\}$ the interval
$I_a=\{e^{i\t}:|\arg(a e^{-i\t})|\le\frac{1-|a|}{2}\}$, and denote
$S(a)=S(I_a)$.

Recall that $\DD$ denotes the class of radial weights such that $\widehat{\om}(r)=\int_r^1\om(s)\,ds$ is doubling, that is, there exists $C=C(\om)\ge1$ such that $\widehat{\om}(r)\le C\widehat{\om}(\frac{1+r}{2})$ for all $0\le r<1$. The following lemma contains basic properties of these weights and will be frequently used in the sequel. The proof is elementary and therefore omitted.

\begin{lemma}\label{Lemma:replacement-Lemma1.1-2.}
Let $\om$ be a radial weight. Then the following conditions are equivalent:
\begin{itemize}
\item[\rm(i)] $\om\in\DD$;
\item[\rm(ii)] There exist $C=C(\om)>0$ and $\b_0=\b_0(\om)>0$ such that
    \begin{equation}\label{Eq:replacement-Lemma1.1}
    \begin{split}
    \widehat{\om}(r)\le C\left(\frac{1-r}{1-t}\right)^{\b}\widehat{\om}(t),\quad 0\le r\le t<1,
    \end{split}
    \end{equation}
for all $\b\ge\b_0$;
\item[\rm(iii)] There exist $C=C(\om)>0$ and $\gamma_0=\gamma_0(\om)>0$ such that
    \begin{equation}\label{Eq:replacement-Lemma1.2}
    \begin{split}
    \int_0^t\left(\frac{1-t}{1-s}\right)^\g\om(s)\,ds
    \le C\widehat{\om}(t),\quad 0\le t<1,
    \end{split}
    \end{equation}
for all $\g\ge\g_0$.
\end{itemize}
Moreover, if $\om\in\DD$, there exists $\lambda_0=\lambda_0(\om)\ge0$ such that
    $$
    \int_\D\frac{\om(z)}{|1-\overline{\z}z|^{\lambda+1}}\,dA(z)\asymp\frac{\widehat{\om}(\zeta)}{(1-|\z|)^\lambda},\quad \z\in\D,
    $$
for all $\lambda>\lambda_0$, and
    $$
    \om(S(z))\asymp\om(T(z))\asymp\om^\star(z),\quad |z|\to1^-,
    $$
where
    $$
    \omega^\star(z)=\int_{|z|}^1\omega(s)\log\frac{s}{|z|}s\,ds,\quad z\in\D\setminus\{0\}.
    $$
\end{lemma}

The classes $\I$ and $\R$ of positive weights considered in \cite{PelRat} are contained in $\DD$, and therefore Lemma~\ref{Lemma:replacement-Lemma1.1-2.} contains most of the assertions in Lemmas~1.1, 1.2, 1.6 and 2.3 from \cite{PelRat} as special cases.

The following result plays a fundamental role in this study. It can be proved also by employing the method used by Sawyer~\cite{Sawyer}.

\begin{theorem}\label{co:maxbou}
Let $0<p\le q<\infty$ and $0<\alpha<\infty$ such that $p\alpha>1$.
Let $\om\in\DD$ and $\mu$ be a positive Borel measure on $\D$. Then
$[M_{\om}((\cdot)^{\frac{1}{\alpha}})]^{\alpha}:L^p_\omega\to
L^q(\mu)$ is bounded if and only if $M_{\om,q/p}(\mu)\in L^\infty$.
Moreover,
    $$
    \|[M_{\om}((\cdot)^{\frac{1}{\alpha}})]^{\alpha}\|^q_{\left(L^p_\om,L^q(\mu)\right)}\asymp\|M_{\om,q/p}(\mu)\|_{L^\infty}.
    $$
\end{theorem}

\begin{proof} The proof of \cite[Corollary~2.2]{PelRat} shows that $[M_{\om}((\cdot)^{\frac{1}{\alpha}})]^{\alpha}:L^p_\omega\to
L^q(\mu)$ is bounded if $M_{\om,q/p}(\mu)\in L^\infty$, and the operator norm has the claimed upper bound. To see the converse, assume that the operator is bounded. Let $S$ be a Carleson square and $f\in L^p_\omega$ such that $\int_S |f|^p \om dA>0$. If $\lambda<\frac{1}{\om(S)} \int_S |f|^p \om dA$,
then $S\subset \left\{z: M_\om(|f|^p\chi_S)(z)>\lambda \right\}$, and by denoting $\phi^{1/\alpha}=|f|^p$, we obtain
    \begin{equation*}
    \begin{split}
    \mu(S) &\le \mu\left( \left\{z: M^{\alpha}_\om(\phi^{1/\alpha}\chi_S)(z)>\lambda^{\alpha} \right\} \right)
    \le\frac{1}{\lambda^{q\alpha}}\int_{\D} M^{q\alpha}_\om(\phi^{1/\alpha}\chi_S)(z)\,d\mu(z)\\
    &\le\frac{\|[M_{\om}((\cdot)^{\frac{1}{\alpha}})]^{\alpha}\|^q_{\left(L^p_\om,L^q(\mu)\right)}}{\lambda^{q\alpha}}
    \|\varphi\chi_S\|_{L^p_\om}^q\\
    &=\frac{\|[M_{\om}((\cdot)^{\frac{1}{\alpha}})]^{\alpha}\|^q_{\left(L^p_\om,L^q(\mu)\right)}}{\lambda^{q\alpha}}
    \left(\int_{S}|f(z)|^{\a p^2}\,\om(z)dA(z)\right)^{q/p}.
    \end{split}
    \end{equation*}
By letting $\lambda \to  \frac{1}{\om(S)} \int_S |f| \om dA$, we deduce
    \begin{equation*}
    \begin{split}
    &\mu(S)\left(\frac{1}{\om(S)} \int_S |f(z)|^p\om(z)\,dA(z)\right)^{q\alpha}\\
    &\quad \le\|[M_{\om}((\cdot)^{\frac{1}{\alpha}})]^{\alpha}\|^q_{\left(L^p_\om,L^q(\mu)\right)}
    \left(\int_{S} |f(z)|^{\a p^2}(z)\om(z)\,dA(z)\right)^{q/p},\quad f\in L^p_\om,
    \end{split}
    \end{equation*}
and the assertion follows by choosing $f\equiv1$.
\end{proof}

\section{Tent spaces}\label{tentspaces}

We defined the tent space $T^p_q(\nu,\om)$ by using the lenses $\Gamma(\z)$ that induce the tents $T(z)$ in a natural manner. One could certainly use different kind of non-tangential approach regions and they would induce the same spaces. In particular, the aperture does not play role in the definition. We will come back to this matter more rigorously later. To make things more precise, we define for each $0<\alpha<\pi$, the lens
    \begin{equation*}
    \Gamma_\alpha(\z)=\left\{z\in \D:\,|\t-\arg
    z|<\alpha\left(1-\frac{|z|}{r}\right)\right\},\quad
    \z=re^{i\theta}\in\overline{\D}\setminus\{0\},
    \end{equation*}
and the tent $T_\alpha(z)=\left\{\z\in\D:\,z\in\Gamma_\alpha(\z)\right\}$. If $\alpha=\frac{1}{2}$, we write $\Gamma(\z)$ and $T(z)$ as before. Moreover, we set $\om(T_\a(0))=\lim_{r\to0^+}\om(T_\a(r))$ to make $\om(T_\a(z))$ well defined and continuous in the whole $\D$.

\subsection{Duality - Part 1}

In this section we prove the most natural duality results for the tent spaces. The result and its proof are useful for our purposes.

\begin{theorem}\label{Thm:tent-spaces-duality}
Let $1\le p,q<\infty$ with $p+q\ne2$, $\om\in\DD$ and let $\nu$ be a positive Borel measure on~$\D$, finite on compact sets of $\D$. Then the dual of $T^p_q(\nu,\om)$ can be identified with $T^{p'}_{q'}(\nu,\om)$ (up to an equivalence of norms) under the pairing
    \begin{equation}\label{Eq:duality-pairing-tent-spaces}
    \langle f,g \rangle_{T^2_2(\nu,\om)}=\int_\D f(z)\overline{g(z)}\om(T(z))\,d\nu(z).
    \end{equation}
\end{theorem}

\begin{proof}
Fubini's theorem and H\"older's inequality imply that each $g\in T^{p'}_{q'}(\nu,\om)$ induces a continuous linear functional on $T^{p}_{q}(\nu,\om)$ under the pairing \eqref{Eq:duality-pairing-tent-spaces} when $1\le q<\infty$ and $1<p<\infty$.

Let now $p=1$ and $1<q<\infty$. For $\z\in\D\setminus\{0\}$ and $0\le h\le\infty$, let
    \begin{equation*}
    \begin{split}
    \Gamma^h(\z)&=\Gamma(\z)\setminus \overline{D\left(0,\frac{|\z|}{1+h}\right)}\\
    &=\left\{z\in\D:|\arg z-\arg \z|<\frac12\left(1-\left|\frac{z}{\z}\right|\right)<\frac{h}{2(1+h)}\right\}
    \end{split}
    \end{equation*}
and
    $$
    A^{q'}_{q',\nu}(g|h)(\z)=\int_{\Gamma^h(\z)}|g(z)|^{q'}\,d\nu(z),\quad \z\in\D\setminus\{0\}.
    $$
Then $\Gamma^0(\z)=\emptyset$ and $\Gamma^\infty(\z)=\Gamma(\z)$ for all $\z\in\D\setminus\{0\}$, and $A_{q',\nu}(g|h)(\z)$ is a non-decreasing function of $h$. For every $g\in T^\infty_{q'}(\nu,\om)$ and $\z\in\D\setminus\{0\}$, define the stopping time by
    $$
    h(\z)=\sup\left\{h:A_{q',\nu}(g|h)(\z)\le C_1C_{q',\nu}(g)(\z)\right\},
    $$
where $C_1>0$ is a large constant to be determined later. Assume for a moment that there exists a constant $C_2>0$ such that
    \begin{equation}\label{Eq:stopping-time}
    \int_\D k(z)\om(T(z))\,d\nu(z)\le C_2\int_\D\left(\int_{\Gamma^{h(\z)}(\z)}k(z)\,d\nu(z)\right)\om(\z)\,dA(\z)
    \end{equation}
for all $\nu$-measurable non-negative functions $k$. Then, by choosing $k(z)=|f(z)||g(z)|$ and applying H\"older's inequality, we obtain
    \begin{equation}\label{Eq:CMS-estimate}
    \begin{split}
    |\langle f,g\rangle_{T^2_2(\nu,\om)}|
    \le C_1C_2\int_\D A_{q,\nu}(f)(\z) C_{q',\nu}(g)(\z)\,\om(\z)dA(\z)
    \lesssim\|f\|_{T^1_q(\nu,\om)}\|g\|_{T^\infty_{q'}(\nu,\om)}.
    \end{split}
    \end{equation}
It remains to prove \eqref{Eq:stopping-time}. Fubini's theorem yields
    $$
    \int_\D\left(\int_{\Gamma^{h(\z)}(\z)}k(z)\,d\nu(z)\right)\om(\z)\,dA(\z)
    =\int_\D\left(\int_{T(z)\cap H(z)}\om(\z)\,dA(\z)\right)k(z)\,d\nu(z),
    $$
where $H(z)=\{\z\in\D:|\z|\le(1+h(\z))|z|\}$, so it suffices to show that
    \begin{equation}\label{Eq:stopping-time-set-estimate}
    \frac1{\om(T(z))}\int_{T(z)\cap H(z)}\om(\z)\,dA(\z)\ge\frac{1}{C_2}
    \end{equation}
for all $z\in\D$. We will prove this only for $z$ close enough to the boundary $\T$, the proof for other values of $z$ follows from this reasoning with appropriate modifications. To see \eqref{Eq:stopping-time-set-estimate}, denote $T^h(z)=\{\z\in\D:z\in\Gamma^{h}(\z)\}=T(z)\cap D(0,(1+h)|z|)$, and for $|z|\ge\frac56$, set $z'=(1-5(1-|z|))z/|z|$ and $x=\frac1{|z|}-1$. Then, in order to have
    $
    T(z)\cap T^x(u)=T(z)\cap T(u)\cap D\left(0,|u|/|z|\right)\ne\emptyset,
    $
we must require $|u|\ge|z|^2$, that is, $1-|u|\le1-|z|^2\le2(1-|z|)$. This together with Fubini's theorem and Lemma~\ref{Lemma:replacement-Lemma1.1-2.} gives
    \begin{equation}\label{st1}
    \begin{split}
    &\frac1{\om(T(z))}\int_{T(z)}\left(\int_{\Gamma^x(\z)}|g(u)|^{q'}\,d\nu(u)\right)\om(\z)\,dA(\z)\\
    &=\frac1{\om(T(z))}\int_{\D}\left(\int_{T(z)\cap T^x(u)}\om(\z)\,dA(\z)\right)|g(u)|^{q'}\,d\nu(u)\\
    &\le\frac1{\om(T(z))}\int_{T(z')}\left(\int_{T(u)}\om(\z)\,dA(\z)\right)|g(u)|^{q'}\,d\nu(u)\\
    &\le \frac{C_3}{\om(T(z'))}\int_{T(z')}|g(u)|^{q'}\om(T(u))\,d\nu(u)
    \le C_3\inf_{v\in T(z)}C^{q'}_{q',\nu}(g)(v),
    \end{split}
    \end{equation}
where the  last inequality is valid because
    $$
    \frac1{\om(T(z'))}\int_{T(z')}|g(u)|^{q'}\om(T(u))\,d\nu(u)
    \le \sup_{a\in\Gamma(v)}\frac1{\om(T(a))}\int_{T(a)}|g(u)|^{q'}\om(T(u))\,d\nu(u)
    $$
for all $v\in T(z)$. Denote $E(z)=\D\setminus H(z)=\{\z\in\D:|\z|>(1+h(\z))|z|\}$. By the definition of $h(\z)$ and \eqref{st1}, and by choosing $C_1$ sufficiently large so that $C_1^{q'}>2C_3$, we deduce
    \begin{equation*}
    \begin{split}
    \int_{T(z)\cap E(z)}\om(\z)\,dA(\z)
    &\le\int_{T(z)\cap E(z)}\frac{A^{q'}_{q',\nu}\left(g\big|\left|\frac{\z}{z}\right|-1\right)(\z)}{C_1^{q'}C^{q'}_{q',\nu}(g)(\z)}\om(\z)\,dA(\z)\\
    &\le\int_{T(z)}\frac{A^{q'}_{q',\nu}(g|x)(\z)}{C_1^{q'}C^{q'}_{q',\nu}(g)(\z)}\om(\z)\,dA(\z)\\
    &\le\frac{1}{C_1^{q'}\inf_{v\in T(z)}C^{q'}_{q',\nu}(g)(v)}\int_{T(z)}A^{q'}_{q',\nu}(g|x)(\z)\om(\z)\,dA(\z)\\
    &\le\frac{C_3\om(T(z))}{C_1^{q'}}<\frac12\om(T(z)).
    \end{split}
    \end{equation*}
Therefore,
    \begin{equation*}
    \begin{split}
    &\frac1{\om(T(z))}\int_{T(z)\cap H(z)}\om(\z)\,dA(\z)
    =1-\frac1{\om(T(z))}\int_{T(z)\cap E(z)}\om(\z)\,dA(\z)\ge\frac12,\quad |z|\ge\frac56,
    \end{split}
    \end{equation*}
and the inequality \eqref{Eq:stopping-time-set-estimate} follows.


Conversely, let $L$ be a continuous linear functional on $T^p_q(\nu,\om)$. Let $L^pL^q(\nu,\om)$ be the space of those $h:\D^2\to\C$ such that
    \begin{equation*}
    \begin{split}
    \|h\|_{{L^{p}L^{q}}(\nu,\om)}^p&=\int_\D\left(\int_\D|h(z,\z)|^q\,d\nu(z)\right)^\frac{p}{q}\om(\z)\,dA(\z)<\infty,\quad 0<p,q<\infty,\\
    \|h\|_{{L^{p}L^{\infty}}(\nu,\om)}^p&=\int_\D\left(\nu\textrm{-ess}\sup_{z\in\D}|h(z,\z)|\right)^p\om(\z)\,dA(\z)<\infty,\quad 0<p<\infty,\\
    \|h\|_{{L^{\infty}L^{q}}(\nu,\om)}^q&=\sup_{\z\in\D}\int_\D|h(z,\z)|^q\,d\nu(z)<\infty\quad 0<q<\infty,
    \end{split}
    \end{equation*}
is well defined and finite. By the Hahn-Banach theorem each bounded linear functional $L:T^p_q(\nu,\om)\to \C $ can be extended to $L^pL^q(\nu,\om)$ with the same norm. We will denote this extension with the same symbol $L$. Further, \cite[Theorem~1.a) on p.~304]{BenedekPanzone} shows that $(L^pL^q(\nu,\om))^\star\simeq L^{p'}L^{q'}(\nu,\om)$ under the pairing
    $$
    \langle f,g\rangle=\int_\D\int_\D f(z,\z)\overline{g(z,\z)}\,d\nu(z)\om(\z)\,dA(\z),\quad 1\le p,q<\infty.
    $$
That is, there exists a function $g\in L^{p'}L^{q'}(\nu,\om)$ such that
    $$
    L(f)=\int_\D\int_\D f(z,\z)\overline{g(z,\z)}\,d\nu(z)\om(\z)\,dA(\z),\quad f\in L^{p}L^{q}(\nu,\om).
    $$
In particular,
    $$
    L(f)=\int_\D\left(\int_{\Gamma(\z)}\overline{g(z,\z)}f(z)\,d\nu(z)\right)\om(\z)\,dA(\z)
    $$
for all $f\in T^p_q(\nu,\om)$ with $\|L\|=\|g\|_{L^{p'}L^{q'}(\nu,\om)}$. Let
    $$
    P_0(g)(z)=\frac1{\om(T(z))}\int_{T(z)}g(z,\z)\om(\z)\,dA(\z),\quad z\in\D\setminus\{0\}.
    $$
Then Fubini's theorem shows that
    \begin{equation}
    \begin{split}\label{P0}
    L(f)&=\int_\D f(z)\left(\frac1{\om(T(z))}\int_{T(z)}g(z,\z)\om(\z)\,dA(\z)\right)\om(T(z))\,d\nu(z)\\
    &=\langle f,P_0(g) \rangle_{T^2_2(\nu,\om)}, \quad f\in T^p_q(\nu,\om).
    \end{split}
    \end{equation}
Therefore it suffices to show that $P_0:L^{p'}L^{q'}(\nu,\om)\to T^{p'}_{q'}(\nu,\om)$ is bounded when $1<p',q'\le\infty$ and $p'$ and $q'$ are not equal to infinity at the same time, that is, $p+q\ne2$.

First suppose that $q'=\infty$ and $1<p'<\infty$. Let $h(u)=\nu\textrm{-ess}\sup_{z\in\D}|g(z,u)|$ for $u\in\D$, and let $g\in L^{p'}L^{\infty}(\nu,\om)$. Then Theorem~\ref{co:maxbou} yields
    \begin{equation*}
    \begin{split}
    \|P_0(g)\|_{T^{p'}_\infty(\nu,\om)}^{p'}&
    \le\int_\D\left(\nu\textrm{-ess}\sup_{z\in\Gamma(\z)}\frac{1}{\om(T(z))}\int_{T(z)}h(u)\om(u)\,dA(u)\right)^{p'}\om(\z)\,dA(\z)\\
    &\lesssim\int_\D M^{p'}_\om(h)(\z)\om(\z)\,dA(\z)\lesssim\|h\|_{L^{p'}_\om}^{p'}=\|g\|_{L^{p'}L^{\infty}(\nu,\om)}^{p'},
    \end{split}
    \end{equation*}
and thus $P_0:L^{p'}L^{\infty}(\nu,\om)\to T^{p'}_{\infty}(\nu,\om)$ is bounded, or equivalently, $P_0\chi_{\Gamma(\z)}$ is bounded on $L^{p'}L^{\infty}(\nu,\om)$.

Now suppose that $1<p'=q'<\infty$. Fubini's theorem and H\"older's inequality give
    \begin{equation*}
    \begin{split}
    \|P_0(g)\|_{T^{p'}_{p'}(\nu,\om)}^{p'}
    &=\int_\D\left|\frac{1}{\om(T(z))}\int_{T(z)}g(z,\z)\om(\z)\,dA(\z)\right|^{p'}\om(T(z))\,d\nu(z)\\
    &\le\int_\D\int_{T(z)}|g(z,\z)|^{p'}\om(\z)\,dA(\z)\,d\nu(z)\\
    &=\|g\chi_{\Gamma(\z)}\|_{L^{p'}L^{p'}(\nu,\om)}^{p'}
    \le \|g\|_{L^{p'}L^{p'}(\nu,\om)}^{p'},\quad 1<p'<\infty.
    \end{split}
    \end{equation*}
By combining this with the previous case, we deduce that $P_0\chi_{\Gamma(\z)}$ defines a bounded operator from $L^{p'}L^{q'}(\nu,\om)$ into itself when either $q'=\infty$ and $1<p'<\infty$ or $1<p'=q'<\infty$. It then follows by \cite[Theorem~2 on p.~316]{BenedekPanzone} that $P_0\chi_{\Gamma(\z)}$ is bounded on all $L^{p'}L^{q'}(\nu,\om)$ when $1<p'\le q'\le\infty$ and $p+q\ne2$. Further, $\chi_{\Gamma(\z)}P_0$ is self-adjoint, and therefore $\chi_{\Gamma(\z)}P_0$ is bounded from $L^{p'}L^{q'}(\nu,\om)$ into itself also when $1< q'\le p'<\infty$.

It still remains to deal with the case $p=1$ and $1<q<\infty$. In this case we again have to show that $P_0:L^{\infty}L^{q'}(\nu,\om)\to T^{\infty}_{q'}(\nu,\om)$ is bounded. To see this, we use H\"older's inequality, Fubini's theorem and Lemma~\ref{Lemma:replacement-Lemma1.1-2.} to deduce
    \begin{equation}\label{1111}
    \begin{split}
    &\frac{1}{\om(T(a))}\int_{T(a)}|P_0(g)(z)|^{q'}\om(T(z))\,d\nu(z)\\
    &\le\frac{1}{\om(T(a))}\int_{T(a)}\int_{T(z)}|g(z,\z)|^{q'}\om(\z)\,dA(\z)\,d\nu(z)\\
    &=\frac{1}{\om(T(a))}\int_{\D}\int_{T(a)\cap\Gamma(\z)}|g(z,\z)|^{q'}\,d\nu(z)\om(\z)\,dA(\z)\\
    &\le\frac{1}{\om(T(a))}\int_{T(a')}\int_{\D}|g(z,\z)|^{q'}\,d\nu(z)\om(\z)\,dA(\z)\\
    &\lesssim\|g\|_{L^\infty L^{q'}(\nu,\om)}^{q'},\quad a'=(1-2(1-|a|))a/|a|,\quad |a|>\frac12,
    \end{split}
    \end{equation}
and it follows that $P_0:L^{\infty}L^{q'}(\nu,\om)\to T^{\infty}_{q'}(\nu,\om)$ is bounded.
\end{proof}

\subsection{Factorization}

In the proof of Theorem~\ref{Thm:tent-spaces-duality} we used a stopping time to obtain the crucial estimate \eqref{Eq:CMS-estimate} which yields
$\langle|f|,|g|\rangle_{T^2_2(\nu,\om)}\lesssim\|A_{q,\nu}(f)C_{q',\nu}(g)\|_{L^1_\om}$ for all $1<q\le\infty$.
If we now replace $f$ by $f^q$, and choose $g(z)=(\om(T(z)))^{-1}$, we deduce
    \begin{equation}
    \begin{split}\label{Eq:CMS-estimate-weak}
    \|f\|_{L^q(\nu)}^q\lesssim\|A^q_{\infty,\nu}(f) M_\om(\nu)\|_{L^1_\om}\le\|M_\om(\nu)\|_{L^\infty}\|f\|_{T^q_\infty(\nu,\om)}^q
    \end{split}
    \end{equation}
and similarly,
    \begin{equation}\label{Eq:CMS-estimate-weak-p>q}
    \begin{split}
    \|f\|_{L^q(\nu)}^q\lesssim\|M_\om(\nu)\|_{L_\om^\frac{p}{p-q}}^\frac{p-q}{p}\|f\|_{T^p_\infty(\nu,\om)}^q,\quad p>q.
    \end{split}
    \end{equation}
If $f$ is analytic in $\D$, then the estimates \eqref{Eq:CMS-estimate-weak} and \eqref{Eq:CMS-estimate-weak-p>q} together with \cite[Lemma~4.4]{PelRat} prove the implication (iii)$\Rightarrow$(i) in Theorem~\ref{Theorem:CarlesonMeasures}(a) and the sufficiency in Theorem~\ref{Theorem:CarlesonMeasures}(b). Carleson measures are discussed in Section~\ref{Section:Carleson}, where alternate proofs of these facts are presented.

The estimate \eqref{Eq:CMS-estimate-weak} will be used to factorize $T^p_q(\nu,\om)$-functions.
To prove the result, we will need the following lemma which is of the same spirit as the Amar-Bonami~\cite[Lemme~1]{AmBo}. Lemma~\ref{Lemma:AmarBonami} can be proved by arguments similar to those used in \eqref{1111}. The measure $(S\mu)_\psi\mu$ appearing in the statement is a $p$-Carleson measure for $A^p_\om$ by the observation above.

\begin{lemma}\label{Lemma:AmarBonami}
Let $\mu$ be a positive Borel measure on $\D$, $\psi:\D\to[0,\infty)$ a $\mu$-measurable function and $\om\in\DD$. Set $B_{\mu,\psi}(\z)=\int_{\Gamma(\z)}\psi(z)\,d\mu(z)$ and
    $$
    (S\mu)_\psi(z)=\psi(z)\int_{T(z)}\frac{\om(\z)}{B_{\mu,\psi}(\z)}dA(\z),\quad z\in\D\setminus\{0\}.
    $$
Then $M_\om((S\mu)_\psi\mu)\in L^\infty$.
\end{lemma}

If $\psi(z)=(\om(T(z)))^{-1}$ we will simply write $S\mu$ and $B_\mu$ instead of $(S\mu)_\psi$ and~$B_{\mu,\psi}$.

\begin{theorem}\label{Thm:Factorization-tent-spaces}
Let $0<p,q<\infty$, $\om\in\DD$ and let $\nu$ be a positive Borel measure on $\D$, finite on compact sets. If $g\in T^p_{\infty}(\nu,\om)$ and $h\in T^{\infty}_{q}(\nu,\om)$, then $f=gh\in T^p_q(\nu,\om)$ with $\|f\|_{T^p_q(\nu,\om)}\lesssim\|g\|_{T^p_{\infty}(\nu,\om)}\|h\|_{T^\infty_q(\nu,\om)}$.

Conversely, each $f\in T^p_q(\nu,\om)$ can be represented in the form $f=gh$, where $g\in T^p_{\infty}(\nu,\om)$ and $h\in T^{\infty}_{q}(\nu,\om)$ such that $\|g\|_{T^p_{\infty}(\nu,\om)}\asymp\|f\|_{T^p_q(\nu,\om)}$ and $\|h\|_{T^\infty_q(\nu,\om)}\lesssim1$.
\end{theorem}

\begin{proof}
Let first $f=gh$, where $g\in T^p_{\infty}(\nu,\om)$ and $h\in T^{\infty}_q(\nu,\om)$, and write $d\mu(z)=|h(z)|^q\om(T(z))\,d\nu(z)$ for short. If $0<p<q<\infty$, then H\"older's inequality, Fubini's theorem and \eqref{Eq:CMS-estimate-weak} yield
    \begin{equation*}
    \begin{split}
    \|f\|_{T^p_q(\nu,\om)}^p
    &\le\int_\D A_{\infty,\nu}(g)^\frac{(q-p)p}{q}(\z)\left(\int_{\Gamma(\z)}|g(z)|^p|h(z)|^q\,d\nu(z)\right)^\frac{p}{q}\om(\z)\,dA(\z)\\
    &\le\|g\|_{T^p_{\infty}(\nu,\om)}^{\frac{q-p}{q}p}\left(\int_\D|g(z)|^p|h(z)|^q\om(T(z))\,d\nu(z)\right)^\frac{p}{q}\\
    &\lesssim\|g\|_{T^p_{\infty}(\nu,\om)}^{\frac{q-p}{q}p}\|g\|_{T^p_{\infty}(\nu,\om)}^{\frac{p^2}{q}}\|h\|_{T^\infty_q(\nu,\om)}^p
    =\|g\|_{T^p_{\infty}(\nu,\om)}^{p}\|h\|_{T^\infty_q(\nu,\om)}^p<\infty,
    \end{split}
    \end{equation*}
and so $f=gh\in T^p_q(\nu,\om)$ with $\|f\|_{T^p_q(\nu,\om)}\lesssim\|g\|_{T^p_{\infty}(\nu,\om)}\|h\|_{T^\infty_q(\nu,\om)}$. The reasoning in the case $q=p$ is similar but easier; just use Fubini's theorem and \eqref{Eq:CMS-estimate-weak}.

If $0<q<p<\infty$, then $(L^\frac{p}{q}_\om)^\star\simeq L^\frac{p}{p-q}_\om$. Let $\varphi\in L^\frac{p}{p-q}_\om$. Then Fubini's theorem, H\"older's inequality, \eqref{Eq:CMS-estimate-weak} and Theorem~\ref{co:maxbou} yield
    \begin{equation*}
    \begin{split}
    \left|\left\langle\int_{\Gamma(\cdot)}|g(z)|^q|h(z)|^q\,d\nu(z),\varphi\right\rangle_{L^2_\om}\right|
    &\lesssim\int_\D|g(z)|^qM_\om(|\varphi|)(z)\,d\mu(z)\\
    &\le\|g\|_{L^p(\mu)}^q\|M_\om(|\vp|)\|_{L^{\frac{p}{p-q}}(\mu)}\\
    &\lesssim\|g\|_{T^p_{\infty}(\nu,\om)}^{q}\|h\|_{T^\infty_q(\nu,\om)}^q\|\varphi\|_{L^\frac{p}{p-q}_\om},
    \end{split}
    \end{equation*}
and it follows that $f=gh\in T^p_q(\nu,\om)$ with $\|f\|_{T^p_q(\nu,\om)}\lesssim\|g\|_{T^p_{\infty}(\nu,\om)}\|h\|_{T^\infty_q(\nu,\om)}$.

To see the converse, let $f\in T^p_q(\nu,\om)$, and write $f=gh$, where
    $$
    g^s(z)=\frac1{\om(T(z))}\int_{T(z)}M_\om\left(A^s_{q,\nu}(f)\right)(\z)\om(\z)\,dA(\z),\quad 0<s<p.
    $$
Then two applications of Theorem~\ref{co:maxbou} yield
    \begin{equation*}
    \begin{split}
    \|g\|_{T^p_\infty(\nu,\om)}^p
    \lesssim\|M_\om(M_\om(A^s_{q,\nu}(f)))\|_{L^{\frac{p}{s}}_\om}^\frac{p}{s}
    \lesssim\|A^s_{q,\nu}(f)\|_{L^{\frac{p}{s}}_\om}^\frac{p}{s}
    =\|f\|_{T^p_q(\nu,\om)}^p.
    \end{split}
    \end{equation*}
Thus $g\in T^p_\infty(\nu,\om)$ with $\|g\|_{T^p_{\infty}(\nu,\om)}\lesssim\|f\|_{T^p_q(\nu,\om)}$. To complete the proof, we need to show that $h=f/g\in T^\infty_q(\nu,\om)$ with a norm bound independent of $f$ and $g$, because once this is proved, the first part of the theorem ensures $\|f\|_{T^p_q(\nu,\om)}\lesssim\|g\|_{T^p_{\infty}(\nu,\om)}$. Write $d\mu(z)=|f(z)|^q\om(T(z))\,d\nu(z)$, and assume for a moment that
    \begin{equation}\label{Eq:Balayage}
    \begin{split}
    |g(z)|^{-q}\lesssim(S\mu)(z)=\frac{1}{\om(T(z))}\int_{T(z)}\frac{\om(\z)}{B_\mu(\z)}\,dA(\z),\quad z\in\D\setminus\{0\}.
    \end{split}
    \end{equation}
This together with Lemma~\ref{Lemma:AmarBonami} yields
    \begin{equation*}
    \begin{split}
    \int_{T(b)}|h(z)|^q\om(T(z))\,d\nu(z)
    &=\int_{T(b)}|g(z)|^{-q}\,d\mu(z)\\
    &\lesssim\int_{T(b)}(S\mu)(z)\,d\mu(z)\lesssim\om(T(b)),\quad b\in\D\setminus\{0\},
    \end{split}
    \end{equation*}
which implies the desired conclusion.

It remains to prove \eqref{Eq:Balayage}, which is equivalent to
    \begin{equation}\label{Eq:tricky-Holder}
    \left(\int_\D\frac{d\eta(\z)}{A^q_{q,\nu}(f)(\z)}\right)^{-\frac1q}\lesssim\left(\int_\D M_\om(A^s_{q,\nu}(f))(\z)\,d\eta(\z)\right)^\frac1s,
    \end{equation}
where
    $$
    d\eta(\z)=\frac{\chi_{T(z)}(\z)\om(\z)\,dA(\z)}{\om(T(z))}
    $$
is a probability measure. The estimate \eqref{Eq:tricky-Holder} without the maximal function $M_\om$ follows by H\"older's inequality, and since a direct calculation gives
    $$
    \int_\D A^s_{q,\nu}(f)(\z)\,d\eta(\z)\lesssim\int_\D M_\om(A^s_{q,\nu}(f))(\z)\,d\eta(\z),
    $$
we obtain \eqref{Eq:Balayage}.
\end{proof}

\subsection{Atomic decomposition}

If $I\subset\T$ is an arc such that $|I|<1$, then there exists $a_I\in\D\setminus\{0\}$ such that $S(a_I)=S(I)$. The tent $T(I)$ associated with such $I$ is $T(a_I)$. If $|I|\ge1$, we set $T(I)=\cup_{J\subset I,\,|J|<1}T(a_J)\cup\{0\}$.

For $0<p<\infty$, a function $a$ is a $T^p_\infty(\nu,\om)$-atom, if it is supported in some tent $T=T(I)$ and $\nu\textrm{-ess}\sup_{u\in T}|a(u)|\le\om(T)^{-\frac1p}$. It is clear that $\|a\|_{T^p_\infty(\nu,\om)}\lesssim1$.

\begin{theorem}\label{Proposition:atoms}
Let $0<p<\infty$, $\om\in\DD$, and let $\nu$ be a positive Borel measure on~$\D$, finite on compact sets. Then each $f\in T^p_\infty(\nu,\om)$ can be represented in the form $f=\sum_j\lambda_ja_j$, where $a_j$ are $T^p_\infty(\nu,\om)$-atoms such that ${\rm supp}(a_j)\cap{\rm supp}(a_k)=\emptyset$ for $j\ne k$, and $\lambda_j\in\C$ such that $\sum_{j}|\lambda_j|^p\asymp\|f\|^p_{T^p_\infty(\nu,\om)}$.
\end{theorem}

\begin{proof}
For $k\in\mathbb{Z}$ and $f\in T^p_\infty(\nu,\om)$, let $O_k(f)=\{z\in\D:A^p_{\infty,\nu}(f)(z)>2^k\}\subset\D$, and consider the open sets $\widehat{O}_k(f)=\{z/|z|:z\in O_k(f)\}$. Then $O_{k+1}(f)\subset O_k(f)$ for all $k$, and $\cup_k O_k(f)$ contains the support of $f$. If $\widehat{O}_k(f)\ne\T$, then let $\widehat{O}_k(f)=\cup_j I^k_j$ be the Whitney covering of $\widehat{O}_k(f)$. If $\widehat{O}_k(f)=\T$, choose a point $\xi\in\T$ such that $\textrm{dist}(\xi,\D\setminus O_k(f))=\textrm{dist}(\T,\D\setminus O_k(f))$, and consider the unit circle as the set $\{e^{it}:\arg\xi\le t<\arg\xi+2\pi\}$. Divide it into two half circles and divide these in turn into two arcs of equal size. Now divide those two that terminate in $\xi$ again into two arcs of equal size, and continue in this fashion until the length of the two arcs of equal size terminating in $\xi$ is less than $\textrm{dist}(\xi,\D\setminus O_k(f))$. Then the length of the shortest arcs in the obtained covering of $\T$ is at least $\textrm{dist}(\xi,\D\setminus O_k(f))/2$. In this case we also denote the covering by $\widehat{O}_k(f)=\cup_j I^k_j$. Now, for each $k$ and $j$, let $J^k_j\subset\T$ be the arc, concentric with $I^k_j$, such that $|J^k_j|=c|I^k_j|$. Let $U(I^k_j)$ be the triangular set in $\D$ bounded by $I^k_j$ and the line segments joining the end points of $I^k_j$ to the origin. Choose now $c\ge2$ sufficiently large such that $O_k(f)\setminus O_{k+1}(f)=\cup_j\Delta^k_j$, where $\Delta^k_j=T(J^k_j)\cap U(I^k_j)\cap(O_k(f)\setminus O_{k+1}(f))$ are disjoint for a fixed $k$. Now, set $a^k_j=f\chi_{\Delta^k_j}2^{-\frac{k+1}{p}}\om(T(J^k_j))^{-\frac1p}$ and $\lambda^k_j=2^{\frac{k+1}{p}}(\om(T(J^k_j)))^\frac1p$. Then $f=\sum_kf\chi_{O_k(f)\setminus O_{k+1}(f)}=\sum_k\sum_jf\chi_{\Delta^k_j}=\sum_k\sum_j\lambda^k_ja^k_j$. This gives us the desired atomic decomposition if we can show that $a^k_j$ are $T^p_\infty(\nu,\om)$-atoms and $\{\lambda^k_j\}$ satisfies the $\ell^p$-norm estimate. It is clear, by the definition of $a^k_j$, that its support is contained in $T(J^k_j)$ and
    $$
    \nu\textrm{-ess}\sup_{z\in T(J^k_j)}|a^k_j(z)|\le\nu\textrm{-ess}\sup_{z\in T(J^k_j)}\left(\frac{|f(z)|\chi_{O_k(f)\setminus O_{k+1}(f)}(z)}{2^{\frac{k+1}{p}}(\om(T(J^k_j)))^\frac1p}\right)
    \le\frac{1}{(\om(T(J^k_j)))^\frac1p}
    $$
by the definition of the sets $O_k(f)$. Hence $a^k_j$ is a $T^p_\infty(\nu,\om)$-atom. Moreover,
    \begin{equation*}
    \begin{split}
    \sum_k\sum_j|\lambda^k_j|^p&=\sum_k2^{k+1}\sum_j\om(T(J^k_j))\asymp\sum_k2^{k+1}\sum_j\om(T(I^k_j))\\
    &\lesssim\sum_k2^{k+1}\om(O_k(f))\asymp\|f\|^p_{T^p_\infty(\nu,\om)},
    \end{split}
    \end{equation*}
and the desired upper bound for the $\ell^p$-norm of $\{\lambda^k_j\}$ follows. The reasoning above readily gives also the same lower bound.
\end{proof}

A function $a$ is a $T^p_q(\nu,\om)$-atom, if it is supported in some tent $T=T(I)$ and
    $$
    \int_T|a(z)|^q\om(T(z))\,d\nu(z)\le\om(T)^{\frac{p-q}{p}},\quad 0<p\le q<\infty.
    $$
By H\"older's inequality and Fubini's theorem, $\|a\|_{T^p_q(\nu,\om)}\lesssim1$.

\begin{theorem}\label{Theorem:atoms-T^1_r}
Let $0<p\le q<\infty$, $\om\in\DD$ and let $\nu$ be a positive Borel measure on~$\D$, finite on compact sets. Then each $f\in T^p_q(\nu,\om)$ can be represented in the form $f=\sum_j\lambda_ja_j$, where $a_j$ are $T^p_q(\nu,\om)$-atoms such that ${\rm supp}(a_j)\cap{\rm supp}(a_k)=\emptyset$ for $j\ne k$, and $\lambda_j\in\C$ such that $\sum_{j}|\lambda_j|^p\asymp\|f\|^p_{T^p_q}(\nu,\om)$.
\end{theorem}

\begin{proof}
Let $f\in T^p_q(\nu,\om)$. By Theorems~\ref{Thm:Factorization-tent-spaces} and~\ref{Proposition:atoms}, we can write $f=gh$, where $g=\sum_j\lambda_j'a_j\in T^p_\infty(\nu,\om)$, the functions $a_j$ are $T^p_\infty(\nu,\om)$-atoms, $\sum_{j}|\lambda_j'|^p\asymp\|g\|^p_{T^p_\infty}(\nu,\om)\asymp\|f\|^p_{T^p_q(\nu,\om)}$ and $h\in T^\infty_q(\nu,\om)$ with $\|h\|_{T^\infty_q(\nu,\om)}\lesssim1$. It follows that $f=\sum_j\lambda_jb_j$, where $b_j=ha_j\|h\|_{T^\infty_q(\nu,\om)}^{-1}$ is supported in the same tent $T$ as $a_j$, and
    $$
    \int_T|b_j(z)|^q\om(T(z))\,d\nu(z)\le\frac{\|h\|_{T^\infty_q(\nu,\om)}^{-q}}{(\om(T))^\frac{q}{p}}\int_T|h(z)|^q\om(T(z))\,d\nu(z)
    \le\frac1{(\om(T))^{\frac{q-p}{p}}},
    $$
because $h\in T^\infty_q(\nu,\om)$. Thus $b_j$ are $T^p_q(\nu,\om)$-atoms and the assertion is proved.
\end{proof}

\subsection{Duality - Part 2}

We next identify the dual of $T^p_q(\nu,\om)$ for $0<q<1$ when $\nu$ is a certain discrete measure. To do this, we use the atomic decomposition obtained in the previous section, and two auxiliary results. The first one is stated as Lemma~\ref{LemmaGlobalCarleson} and it is implicit in the proof of \cite[Lemma~5.3]{PelRat}.

\begin{lemma}\label{LemmaGlobalCarleson}
Let $0<\a<\infty$, $\omega\in\DD$ and let $\mu$ be a
positive Borel measure on $\D$. Then there exists
$\la=\la(\omega,\a)>1$ such that
    \begin{equation*}
    \begin{split}
    M_{\om,\a}(\mu)(\z)
    \asymp\sup_{a\in\Gamma (\z)}\frac{\mu(S(a))}{(\omega(S(a)))^\a}
    \asymp\sup_{a\in\Gamma (\z)}\frac{1}{(\omega(S(a)))^\a}\int_\D\left(\frac{1-|a|}{|1-\overline{a}z|}\right)^\la d\mu(z),\quad \z\in\D.
    \end{split}
    \end{equation*}
\end{lemma}

The proofs of Theorems~\ref{Thm:tent-spaces-duality} and~\ref{Theorem:atoms-T^1_r} are independent of the aperture $\alpha\in(0,\pi)$ of the lens chosen in the definition of $T^p_q(\nu,\om)$-spaces. This fact will be needed in the proof of the next lemma.

\begin{lemma}\label{Lemma:cone-integral-nu}
Let $0<p<\infty$ and $\om\in\DD$, and let $\nu$ be a positive Borel measure on~$\D$, finite on compact sets. Then there exists $\lambda_0=\lambda_0(p,\om)\ge1$ such that
    \begin{equation}\label{Eq:cone-integral-nu}
    \int_\D\left(\int_\D\left(\frac{1-|z|}{|1-\overline{\z}z|}\right)^\lambda d\nu(z)\right)^p\om(\z)\,dA(\z)\asymp\int_\D\left(\nu(\Gamma(\z))\right)^p\om(\z)\,dA(\z)+\nu(\{0\})
    \end{equation}
for each $\lambda>\lambda_0$.
\end{lemma}

\begin{proof} Without loss of generality we may assume that $\nu(\{0\})=0$. By integrating only over $\Gamma(\z)$ one sees that the right hand side is dominated by a sufficiently large constant, depending on $\lambda$ and $p$, times the left hand side. The opposite inequality will be proved in several steps.

The case $p=1$ follows by Fubini's theorem and Lemma~\ref{Lemma:replacement-Lemma1.1-2.}.

Let $1<p<\infty$ and $g\in L^{p'}_\om\simeq(L^p_\om)^\star$, and write $h_\lambda(z,\z)=\left(\frac{1-|z|}{|1-\overline{\z}z|}\right)^\lambda$ for short. Then, by two applications of Fubini's theorem, H\"older's inequality and Lemma~\ref{LemmaGlobalCarleson}, we get
    \begin{equation*}
    \begin{split}
    \left|\left\langle g,\int_\D h_\lambda(z,\cdot) d\nu(z)\right\rangle_{L^2_\om}\right|
    &\le\int_\D\sup_{z\in\Gamma(u)}\left(\frac{\int_\D|g(\z)|h_\lambda(z,\z) \om(\z)\,dA(\z)}{\om(T(z))}\right)\nu(\Gamma(u))\om(u)\,dA(u)\\
    &\lesssim\left(\int_\D\left(\nu(\Gamma(u))\right)^p\om(u)\,dA(u)\right)^\frac1{p}
    \|M_\om(g)\|_{L^{p'}_\om}.
    \end{split}
    \end{equation*}
Therefore Theorem~\ref{co:maxbou} yields
    \begin{equation*}
    \begin{split}
    \left|\left\langle g,\int_\D h_\lambda(z,\cdot) d\nu(z)\right\rangle_{L^2_\om}\right|
    &\lesssim\left(\int_\D\left(\nu(\Gamma(u))\right)^p\om(u)\,dA(u)\right)^\frac1{p}\|g\|_{L_\om^{p'}},
    \end{split}
    \end{equation*}
and the assertion for $1<p<\infty$ follows by the duality.

Let now $0<p<1$. Let $\r\in(0,1)$ to be fixed later and define $\r(z)=\r|z|(1-|z|)$ for all $z\in\D$.
Since $h_\lambda(z_1,\z)\asymp h_\lambda(z_2,\z)$ for all $\z\in \D$ whenever $z_1,z_2$ belong to a hyperbolically bounded region,
Fubini's theorem yields
  \begin{equation}\begin{split}\label{eq:j1}
    \int_\D h_\lambda(z,\z)\,d\nu(z) &=\int_\D\left(\int_\D\frac{h_\lambda(z,\z)}{|D(z,\r(z))|} \chi_{D(z,\r(z))}(u)d\nu(z) \right)\,dA(u)
    \\ &  \asymp  \int_\D F(u,\r,\nu)h_\lambda(u,\z)\,dh(u),
    \end{split}\end{equation}
where
    $$
    F(u,\r,\nu)=\int_\D\frac{\chi_{D(z,\r(z))}(u)}{\r^2|z|^2}\,d\nu(z).
    $$
Here $\chi_{D(z,\r(z))}(u)$ denotes the function that is equal to $1$ if $u\in D(z,\r(z))$ and is 0 otherwise.
Let now $\alpha\in\left(0,\frac12\right)$ be fixed and choose $\r$ small enough such that $\Gamma_\a(\z)\cap D(z,\r(z))=\emptyset$ for
all $z\in \D\setminus \Gamma\left(\frac{1+|\z|}{2}e^{i\arg\z}\right)$. Then
Fubini's theorem gives
    \begin{equation}\label{eq:j2}
    \begin{split}
   \int_{\Gamma_\a(\z)}F(u,\r,\nu)\,dh(u)
    \lesssim \nu\left(\Gamma\left(\frac{1+|\z|}{2}e^{i\arg\z}\right)\right).
    \end{split}
    \end{equation}
Assume without loss of generality that $\int_0^1\om(r)\,dr=1$. Let $C$ and $\b$ be those from Lemma~\ref{Lemma:replacement-Lemma1.1-2.}(ii), and define $r_n$ by $\widehat{\om}(r_n)=(2C)^{-n}$ for $n\in\N\cup\{0\}$. Then
    $$
    2C=\frac{\widehat{\om}(r_n)}{\widehat{\om}(r_{n+1})}\le C\left(\frac{1-r_n}{1-r_{n+1}}\right)^{\b}
    $$
for all $n$. Hence
$1-r_{n+1}\le \frac{1-r_n}{2^{\frac{1}{\b}}}$, and therefore for $N=\min\{k\in\N:k>\b+1\}$ we have
    $
    1-r_{n+N}\le2^{\frac{1-N}{\b}}(1-r_{n+1})\le (1-r_{n+1})/2.
    $
This implies $\frac{1+r_{n+1}}{2}\le r_{n+N}$ for $n\in\N\cup\{0\}$, 
and hence, by Lemma~\ref{Lemma:replacement-Lemma1.1-2.},
    \begin{equation}\label{eq:j3}
    \begin{split}
    &\int_\D\left(\nu\left(\Gamma\left(\frac{1+|\z|}{2}e^{i\arg\z}\right)\right)\right)^p\om(\z)\,dA(\z)\\
    &\lesssim\int_0^{2\pi}\sum_{n=0}^\infty\left(\nu\left(\Gamma\left(\frac{1+r_{n+1}}{2}e^{i\t}\right)\right)\right)^p\int_{r_n}^{r_{n+1}}\om(r)\,dr\,d\theta\\
   &\lesssim\int_0^{2\pi}\sum_{n=0}^\infty\left(\nu\left(\Gamma\left(r_{n+N}e^{i\t}\right)\right)\right)^p\int_{r_{n+N}}^{r_{n+N+1}}\om(r)\,dr\,d\theta
   \le\int_\D\nu\left(\Gamma\left(\z\right)\right)^p\om(\z)\,dA(\z).
    \end{split}
    \end{equation}
Therefore it suffices to show that
    \begin{equation}
    \begin{split}\label{eq:j4}
    &\int_\D\left(\int_\D F(z,\r,\nu)h_\lambda(z,\z)\,dh(z)\right)^p\om(\z)\,dA(\z)\\
    &\lesssim\int_\D\left(\int_{\Gamma_\alpha(\z)}F(z,\r,\nu)\,dh(z)\right)^p\om(\z)\,dA(\z),
    \end{split}
    \end{equation}
because by combining this with \eqref{eq:j1}, \eqref{eq:j2} and \eqref{eq:j3}, we deduce
    \begin{equation}\label{4}
    \begin{split}
    &\int_\D\left( \int_\D h_\lambda(z,\z) d\nu(z)\right)^p\om(\z)\,dA(\z)
    \\ & \asymp \int_\D\left(\int_\D F(z,\r,\nu)h_\lambda(z,\z)\,dh(z)\right)^p\om(\z)\,dA(\z)\\
    &\lesssim\int_\D\left(\int_{\Gamma_\alpha(\z)}F(z,\r,\nu)\,dh(z)\right)^p\om(\z)\,dA(\z)\\
    &\lesssim\int_\D\left(\nu\left(\Gamma\left(\frac{1+|\z|}{2}e^{i\arg\z}\right)\right)\right)^p\om(\z)\,dA(\z)
    \\  & \lesssim \int_\D\nu\left(\Gamma\left(\z\right)\right)^p\om(\z)\,dA(\z).
    \end{split}
    \end{equation}
Writing $g^q(z)=F(z,\r,\nu)$, where $q=1/p\in(1,\infty)$, \eqref{eq:j4} becomes
 \begin{equation}\label{3.3}
    \int_\D\left(\int_\D g^q(z)h_\lambda(z,\z)\,dh(z)\right)^{1/q}\om(\z)\,dA(\z)\lesssim\|g\|_{T^1_q\left(h,\om\right)},
    \end{equation}
where the norm in $T^1_q\left(h,\om\right)$ is taken with respect to lenses $\Gamma_\alpha(\cdot)$. We will show that \eqref{3.3} is actually valid for all measurable functions $g$.

As observed before, in Theorem~\ref{Theorem:atoms-T^1_r} the aperture $\a$ does not play any role, and hence it suffices to show that the left hand side of \eqref{3.3} is bounded for a $T^1_q(h,\om)$-atom $b$ supported in a tent $T=T(I)$. Let $S=S(I)\supset T(I)$.
We split the outer integral into two parts; $2S$ and $\D\setminus 2S$. By H\"{o}lder's inequality, Fubini's theorem, Lemma~\ref{Lemma:replacement-Lemma1.1-2.}, and the definition of $T^1_q(h,\om)$-atoms, we deduce
    \begin{equation*}
    \begin{split}
    &\int_{2S}\left(\int_{S} |b(z)|^qh_\lambda(z,\z)\,dh(z)\right)^{\frac1q}\om(\z)\,dA(\z)\\
    &\lesssim(\om(T))^{\frac1{q'}}\left(\int_{2S}\int_{S}|b(z)|^q h_\lambda(z,\z)\,dh(z)\om(\z)\,dA(\z)\right)^{\frac1{q}}\\
    &\le(\om(T))^{\frac1{q'}}\left(\int_{S}\left(\int_{\D} h_\lambda(z,\z) \om(\z)\,dA(\z)\right)|b(z)|^q\,dh(z)\right)^{\frac1{q}}\\
    &\lesssim(\om(T))^{\frac1{q'}}\left(\int_{T}|b(z)|^q\om(T(z))\,dh(z)\right)^{\frac1{q}}\lesssim1.
    \end{split}
    \end{equation*}
For the integral over $\D\setminus 2S$, take $\b=\log_2C$, where $C=C(\om)$ is the constant from the definition of $\DD$. By Lemma~\ref{Lemma:replacement-Lemma1.1-2.} we may choose $\lambda>(\beta+1)q$ large enough such that $\frac{(1-t)^\lambda}{\om(T(t))}$ is a essentially decreasing on $[0,1)$. This together with standard estimates and the definition of $T^1_q(h,\om)$-atoms gives
    \begin{equation*}
    \begin{split}
    &\int_{\D\setminus 2S}\left(\int_{S}|b(z)|^qh_\lambda(z,\z)\,dh(z)\right)^{\frac1q}\om(\z)\,dA(\z)\\
    &\lesssim\left(\frac{|I|^\lambda}{\om(T)}\right)^{\frac1q}
    \sum_{k=1}^\infty \int_{2^{k+1}S\setminus 2^kS}\left(\int_{T}\frac{|b(z)|^q}{|1-\overline{\z}z|^\lambda} \om(T(z))\,dh(z)\right)^{\frac1q}\om(\z)\,dA(\z)\\
    &\lesssim \frac{1}{\om(T)}
    \sum_{k=1}^\infty \frac{1}{2^{k\frac{\lambda}{q}}}\int_{2^{k+1}S\setminus 2^kS}\om(\z)\,dA(\z)
    \lesssim\sum_{k=1}^\infty 2^{k\left(\beta +1-\frac{\lambda}{q}\right)}<\infty,
    \end{split}
    \end{equation*}
and thus \eqref{3.3} is proved.
\end{proof}

The proof of Lemma~\ref{Lemma:cone-integral-nu} shows that we may replace $\Gamma(\z)$ by $\Gamma_\alpha(\z)$ for any $\alpha\in(0,\pi)$ in \eqref{Eq:cone-integral-nu}. It follows that the space $T^p_q(\nu,\om)$ is independent of the aperture of the lens appearing in the definition, and the quasi-norms obtained for different lenses are equivalent.

If $\nu=\sum_k \d_{z_k}$, where $\{z_k\}$ is a separated sequence, then we write $T^p_q(\nu,\om)=T^p_q(\{z_k\},\om)$.

\begin{proposition}\label{Proposition:duality-sequence-tent}
Let $0<q<1<p<\infty$, $\om\in\DD$ and $\{z_k\}$ be a separated sequence such that $z_k\ne0$ for all $k$. If $L$ is a continuous linear functional on $T^p_q(\{z_k\},\om)$, then there exists a unique $g\in T^{p'}_\infty(\{z_k\},\om)$ such that
    \begin{equation}\label{Eq:duality-pairing-tent-spaces-sequences}
    L(f)=\langle f,g\rangle_{T^2_2(\{z_k\},\om)}=\sum_k f(z_k)\overline{g(z_k)}\om(T(z_k)),\quad f\in T^p_q(\{z_k\},\om),
    \end{equation}
and $\|L\|\asymp\|g\|_{T^{p'}_{\infty}(\{z_k\},\om)}$.

Conversely, each $g\in T^{p'}_\infty(\{z_k\},\om)$ induces a continuous linear functional on $T^p_q(\{z_k\},\om)$ by means of \eqref{Eq:duality-pairing-tent-spaces-sequences}.
\end{proposition}

\begin{proof}
Let $0<q<1<p<\infty$ so that $p'>1$, and let $g\in T^{p'}_\infty(\{z_k\},\om)$. Then Fubini's theorem, the inequality $(\sum a_k)^q\le\sum a_k^q$ and H\"older's inequality yield $|\langle f,g\rangle_{T^2_2(\{z_k\},\om)}|\le\|g\|_{T^{p'}_{\infty}(\{z_k\},\om)}\|f\|_{T^{p}_{q}(\{z_k\},\om)}$.
Therefore each $g\in T^{p'}_\infty(\{z_k\},\om)$ induces a continuous linear functional $L$ on $T^p_q(\{z_k\},\om)$ by means of \eqref{Eq:duality-pairing-tent-spaces-sequences} and $\|L\|\le\|g\|_{T^{p'}_\infty(\{z_k\},\om)}$.

Let now $L$ be a linear continuous  functional on $T^p_q(\{z_k\},\om)$, and consider the sequence $\{b_k\}$ such that $b_k=L(e_k)/\om(T(z_k))$, where $e_k$ is the function with $e_k(z_j)=\d_{jk}$. Let $g$ be such that $\overline{g(z_k)}=b_k$ for all $k$. Write $f\in T^p_q(\{z_k\},\om)$ as a linear combination of $e_k$'s. Then Fubini's theorem gives
    $$
    L(f)=\sum_kf(z_k)L(e_k)=\sum_kf(z_k)b_k\om(T(z_k))=\int_\D\left(\sum_{z_k\in\Gamma(\z)}f(z_k)b_k\right)\om(\z)\,dA(\z).
    $$
We need to show that $g\in T^{p'}_\infty(\{z_k\},\om)$, that is,
    \begin{equation}\label{Eq-1}
    \int_\D\left(\sup_{z_k\in\Gamma(\z)}|b_k|\right)^{p'}\om(\z)\,dA(\z)<\infty.
    \end{equation}
Let $\mathcal{Y}$ denote the set of standard dyadic intervals of $\T$, and consider the family $\{R(I):I\in\mathcal{Y}\}$ of dyadic Carleson top halves. Since $\{z_k\}$ is separated by the assumption, there exists $M\in\N$ such that each dyadic top half $R$ contains at most $M$ points $\{z_k\}$. We may assume $M=1$, for otherwise we just have to multiply our estimates by $M$. Lemma~\ref{Lemma:cone-integral-nu} with $\nu=\sum_k\d_k$ shows that
    $$
    \|f\|_{T^p_q(\{z_k\},\om)}^p\asymp\int_\D\left(\sum_k|f(z_k)|^q\left(\frac{1-|z_k|}{|1-\overline{z}_k\z|}\right)^\lambda\right)^\frac{p}{q}
    \om(\z)\,dA(\z)
    $$
for $\lambda=\lambda(\om)>1$ sufficiently large. Since $h(z)=\frac{1-|z|}{|1-\overline{z}\z|}$ is essentially constant in each top half, we may assume that the points $z_k$ are the centers of $R^k_j=R(I^k_j)$, where $j=0,1,\ldots,2^k-1$ and $k\in\N\cup\{0\}$. We re-index them such a way that $z_{j,k}$ is the center of $R^k_j$, and similarly, we write $b_{j,k}=\overline{g(z_{j,k})}=L(e_{j,k})/\om(T(z_{j,k}))$. Without loss of generality we may assume that $\arg f(z_{j,k})=\arg g(z_{j,k})$ for all $k$ and $j$. Then Fubini's theorem gives
    \begin{equation*}
    \begin{split}
    \langle f,g \rangle_{T^2_2(\{z_k\},\om)}&=\sum_{j,k}|f(z_{j,k})||g(z_{j,k})|\om(T(z_{j,k}))\\
    &=\int_\D\sum_{j,k}|f(z_{j,k})||g(z_{j,k})|\chi_{T(z_{j,k})}(\z)\om(\z)\,dA(\z).
    \end{split}
    \end{equation*}
Put
    $$
    \vp_k(\z)=\sum_{j=0}^{2^k-1}|f(z_{j,k})|\chi_{T(z_{j,k})}(\z),\quad\psi_k(\z)=\sum_{j=0}^{2^k-1}|g(z_{j,k})|\chi_{T(z_{j,k})}(\z),
    $$
so that
    \begin{equation*}
    \begin{split}
    \langle f,g \rangle_{T^2_2(\{z_k\},\om)}&=\int_\D\sum_{k}\vp_k(\z)\psi_k(\z)\om(\z)\,dA(\z).
    \end{split}
    \end{equation*}
It is clear that for any fixed $\z\in\D$, the sequences $\{\vp_k(\z)\}$ and $\{\psi_k(\z)\}$ contain only finitely many elements that are different from $0$. Clearly, for any fixed $k$, $T(z_{j,k})\cap T(z_{m,k})=\emptyset$ for $j\ne m$, and hence
    \begin{equation}\label{Eq:lkj}
    \begin{split}
    \|f\|_{T^p_q(\{z_k\},\om)}^p
    &=\int_\D\left(\sum_{k}\sum_{j=0}^{2^k-1}|f(z_{jk})|^q\chi_{T(z_{jk})}(\z)\right)^\frac{p}{q}\om(\z)\,dA(\z)\\
    &=\int_\D\left(\sum_{k}\left(\sum_{j=0}^{2^k-1}|f(z_{jk})|\right)^q\chi_{T(z_{jk})}(\z)\right)^\frac{p}{q}\om(\z)\,dA(\z)\\
    &=\int_\D\left(\sum_{k}\vp^q_k(\z)\right)^\frac{p}{q}\om(\z)\,dA(\z).
    \end{split}
    \end{equation}
Hence the boundedness of $L$ is equivalent to
    \begin{equation}\label{Eq:boundedness}
    \begin{split}
    \int_\D\sum_{k}\vp_k(\z)\psi_k(\z)\om(\z)\,dA(\z)\lesssim\left(\int_\D\left(\sum_{k}\vp^q_k(\z)\right)^\frac{p}{q}\om(\z)\,dA(\z)\right)^\frac1p
    \end{split}
    \end{equation}
for all $\{\vp_k\}$ for which the last integral in \eqref{Eq:lkj} is finite,
and similarly, \eqref{Eq-1} becomes
    \begin{equation*}
    \begin{split}
    \int_\D\left(\sup_{k}\psi_k(\z)\right)^{p'}\om(\z)\,dA(\z)
    &=\int_\D\left(\sup_k\sum_{j=0}^{2^k-1}|g(z_{jk})|\chi_{T(z_{jk})}(\z)\right)^{p'}\om(\z)\,dA(\z)\\
    &=\int_\D\left(\sup_{z_{jk}\in\Gamma(\z)}|g(z_{jk})|\right)^{p'}\om(\z)\,dA(\z)<\infty.
    \end{split}
    \end{equation*}
To finish the proof we will need the following lemma, see \cite[Lemma~1]{Lu90}.

\begin{lemma}
Let $0<q<\infty$ and $\vp_k$ be a sequence of non-negative functions such that for each fixed $\z\in\D$, the sequence $\{\vp_k(\z)\}$ contains only finitely many members that are different from $0$. For each $\z\in\D$, define
    \begin{equation*}
    \begin{split}
    k_1(\z)&=\min\{k:\vp_k(\z)>0\},\\
    k_{j+1}(\z)&=\min\{k:\vp_k(\z)>2\vp_{k_j}(\z)\},
    \end{split}
    \end{equation*}
and
    \begin{equation*}
    \widetilde{\vp}_k(\z)=\left\{
        \begin{array}{cl}
        \vp_k(\z),& \quad k\in\{k_1,k_2,\ldots\}  \\
        0,&\quad k\not\in\{k_1,k_2,\ldots\}
        \end{array}\right.,\quad \z\in\D.
    \end{equation*}
Then
    \begin{equation}\label{(3.5)}
    \max_k\widetilde{\vp}_k(\z)\le\left(\sum_k\widetilde{\vp}_k^q(\z)\right)^\frac1q\lesssim\max_k\widetilde{\vp}_k(\z),\quad \z\in\D,
    \end{equation}
and
    \begin{equation}\label{(3.6)}
    \max_k\widetilde{\vp}_k(\z)\le\max_k\vp_k(\z)\le2\max_k\widetilde{\vp}_k(\z),\quad \z\in\D.
    \end{equation}
\end{lemma}

Now we may finish the proof of Proposition~\ref{Proposition:duality-sequence-tent}. Choose $\vp_k=\psi_k^{p'-1}$, and note that \eqref{Eq:boundedness} implies
    \begin{equation}\label{Eq:boundedness2}
    \begin{split}
    \int_{D(0,\r)}\sum_{k}\widetilde{\vp}_k(\z)\psi_k(\z)\om(\z)\,dA(\z)\lesssim\left(\int_{D(0,\r)}\left(\sum_{k}\widetilde{\vp}^q_k(\z)\right)^\frac{p}{q}\om(\z)\,dA(\z)\right)^\frac1p
    \end{split}
    \end{equation}
for all $0<\r<1$. Then, for any $\z\in\D$, \eqref{(3.5)} and \eqref{(3.6)} yield
    \begin{equation}\label{(3.8)}
    \begin{split}
    \sum_k\widetilde{\vp}_k(\z)\psi_k(\z)
    &=\sum_k\widetilde{\vp}^p_k(\z)\ge\left(\max_k\widetilde{\vp}_k(\z)\right)^p\\
    &\ge2^{-p}\left(\max_k\vp_k(\z)\right)^p=2^{-p}\max_k\psi_k^{p'}(\z)
    \end{split}
    \end{equation}
and
    \begin{equation}\label{(3.9)}
    \begin{split}
    \left(\sum_k\widetilde{\vp}^q_k(\z)\right)^\frac{p}{q}\lesssim\left(\max_k\widetilde{\vp}_k(\z)\right)^p
    \le\left(\max_k\vp_k(\z)\right)^p=\max_k\psi_k^{p'}(\z).
    \end{split}
    \end{equation}
By combining \eqref{(3.8)}, \eqref{Eq:boundedness2} and \eqref{(3.9)}, we deduce
    \begin{equation*}
    \begin{split}
    \int_{D(0,\r)}\max_{k}\psi^{p'}_k(\z)\om(\z)\,dA(\z)
    &\lesssim\int_{D(0,\r)}\sum_k\widetilde{\vp}_k(\z)\psi_k(\z)\om(\z)\,dA(\z)\\
    &\lesssim\left(\int_{D(0,\r)}\left(\sum_{k}\widetilde{\vp}^q_k(\z)\right)^\frac{p}{q}\om(\z)\,dA(\z)\right)^\frac1p\\
    &\lesssim\left(\int_{D(0,\r)}\max_k\psi_k^{p'}(\z)\om(\z)\,dA(\z)\right)^\frac1p,\quad 0<\r<1,
    \end{split}
    \end{equation*}
and, by letting $\r\to1^-$, we obtain
    $
    \int_\D\max_{k}\psi^{p'}_k(\z)\om(\z)\,dA(\z)\lesssim1,
    $
which is equivalent to \eqref{Eq-1}. Proposition~\ref{Proposition:duality-sequence-tent} is now proved.
\end{proof}

\subsection{Operators acting from $T^p_2(\{z_k\},\om)$ to $A^p_\om$} The following result allows us to construct appropriate test functions in the proofs of Theorems~\ref{Theorem:CarlesonMeasures} and~\ref{Theorem:DifferentiationOperator}. It plays in part a similar role on $A^p_\om$ as the atomic decomposition~\cite{Ro:de} on standard Bergman spaces $A^p_\alpha$.

\begin{lemma}\label{Lemma:operator-tent-bergman}
Let $0<p<\infty$ and $\om\in\DD$, and let $\{z_k\}$ be a separated sequence. Define
    $$
    S_\lambda(f)(z)=\sum_k f(z_k)\left(\frac{1-|z_k|}{1-\overline{z}_kz}\right)^\lambda,\quad z\in\D.
    $$
Then $S_\lambda:T^p_2(\{z_k\},\om)\to A^p_\om$ is bounded for all $\lambda>\lambda_0$, where $\lambda_0=\lambda_0(p,\om)\ge1$ is that of Lemma~\ref{Lemma:cone-integral-nu}.
\end{lemma}

\begin{proof}
To begin with, let us show that $S_\lambda$ is well defined. If $|z|<R<1$, then Lemma~\ref{Lemma:cone-integral-nu} yields
\begin{equation*}\begin{split}
|S_\lambda(f)(z)|^p & \lesssim\left( \sum_k |f(z_k)|^2(1-|z_k|)^{\lambda}\right)^{\frac{p}2}\left( \sum_k (1-|z_k|)^{\lambda}\right)^{\frac{p}2}
\\ & \lesssim \inf_{\z\in\D}\left( \sum_k |f(z_k)|^2\left|\frac{1-|z_k|}{1-\overline{z}_k\z}\right|^\lambda\right)^{\frac{p}2}
\\ & \lesssim \int_\D\left(\sum_k|f(z_k)|^2\left(\frac{1-|z_k|}{|1-\overline{z}_k\z|}\right)^\lambda\right)^\frac{p}{2}\om(\z)\,dA(\z)
\\ & \lesssim \int_\D\left(\sum_{z_k\in\Gamma(\z)}|f(z_k)|^2\right)^\frac{p}{2}\om(\z)\,dA(\z)=\|f\|_{T^p_2(\{z_k\},\om)}^p,
\end{split}\end{equation*}
which in particular implies
    \begin{equation}\label{eq:j5}
    |S_\lambda(f)(0)| \lesssim \|f\|_{T^p_2(\{z_k\},\om)}.
    \end{equation}
Further, by the Cauchy-Schwarz inequality,
    \begin{equation*}
    \begin{split}
    (1-|z|)^2|S_\lambda(f)'(z)|^2&\lesssim(1-|z|)\sum_k|f(z_k)|^2\frac{(1-|z_k|)^\lambda}{|1-\overline{z}_kz|^{\lambda+1}}\\
    &\quad\cdot(1-|z|)\sum_k\frac{(1-|z_k|)^\lambda}{|1-\overline{z}_kz|^{\lambda+1}}=I\cdot II.
    \end{split}
    \end{equation*}
Let $0<r<1$ be fixed such that $\Delta(z_k,r)$ are disjoint. The latter sum $II$ above can be estimated upwards by using the subharmonicity as follows
    \begin{equation*}
    \begin{split}
    II\lesssim(1-|z|)\sum_k\int_{\Delta(z_k,r)}\frac{(1-|\z|)^{\lambda-2}}{|1-\overline{z}\z|^{\lambda+1}}\,dA(\z)
    \le(1-|z|)\int_{\D}\frac{(1-|\z|)^{\lambda-2}}{|1-\overline{z}\z|^{\lambda+1}}\,dA(\z)\asymp1,
    \end{split}
    \end{equation*}
provided $\lambda>1$. Therefore \cite[Theorem~4.2]{PelRat} and the estimate $|1-\overline{w}z|\asymp |1-\overline{w}|z||$ for $z\in\Gamma(1)$ yield
    \begin{equation*}
    \begin{split}
    &\|S_\lambda(f)-S_\lambda(f)(0)\|_{A^p_\om}^p\lesssim
    \int_\D\left(\int_{\Gamma(\z)}\sum_k|f(z_k)|^2\frac{(1-|z_k|)^\lambda}{|1-\overline{z}_kz|^{\lambda+1}}\frac{dA(z)}{1-|z|}\right)^\frac{p}{2}\om(\z)\,dA(\z)\\
    &\asymp\int_\D\left(\sum_k|f(z_k)|^2(1-|z_k|)^\lambda\int_{\Gamma(|\z|)}\frac{dA(z)}{|1-\overline{z}_ke^{i\arg \z}|z||^{\lambda+1}(1-|z|)}\right)^\frac{p}{2}\om(\z)\,dA(\z)\\
    &\lesssim\int_\D\left(\sum_k|f(z_k)|^2(1-|z_k|)^\lambda\int_0^{|\z|}\frac{\,ds}{|1-\overline{z}_ke^{i\arg \z}s|^{\lambda+1}}\right)^\frac{p}{2}\om(\z)\,dA(\z)\\
    &\asymp\int_\D\left(\sum_k|f(z_k)|^2\left(\frac{1-|z_k|}{|1-\overline{z}_k\z|}\right)^\lambda\right)^\frac{p}{2}\om(\z)\,dA(\z).
    \end{split}
    \end{equation*}
Now, by applying Lemma~\ref{Lemma:cone-integral-nu} with $\nu=\sum_k|f(z_k)|^2\delta_{z_k}$, we obtain
    \begin{equation*}
    \begin{split}
    \|S_\lambda(f)-S_\lambda(f)(0)\|_{A^p_\om}^p
    &\lesssim\int_\D\left(\sum_{z_k\in\Gamma(\z)}|f(z_k)|^2\right)^\frac{p}{2}\om(\z)\,dA(\z)=\|f\|_{T^p_2(\{z_k\},\om)}^p,
    \end{split}
    \end{equation*}
and the assertion is proved.
\end{proof}

\section{Carleson measures for $A^p_\om$}\label{Section:Carleson}

\subsection{Case $0<q<p<\infty$}

Apart from the standard dyadic intervals $\mathcal{Y}$ of $\T$, we will consider a family of dyadic tents that is induced by a slightly different covering of $\T$. Let $\Upsilon$ denote the family of dyadic intervals of $\T$ of the form
    $$
    I_{n,k}=\left\{e^{i\theta}:\,\frac{\pi k}{2^{n+2}}\le
    \theta<\frac{\pi(k+1)}{2^{n+2}}\right\},
    $$
where $k=0,1,2,\dots,2^{n+2}-1$ and $n\in\N\cup\{0\}$.
For $I_j\in\Upsilon$, we will write $z_j$ for the
unique point in $\D$ such that $z_j=(1-2|I_j|/\pi)a_j$, where
$a_j\in\T$ is the midpoint of $I_j$. By these definitions, $T(z_{n-1,k})\subset T(z_{n,l})$ whenever $I_{n-1,k}\subset I_{n,l}$.

The dyadic maximal functions of a measure $\mu$ are defined as
    $$
    M^d_\om(\mu)(\zeta)=\sup_{\{S(I),\, I\in\Upsilon:\z\in S(I)\}}\frac{\mu{(S)}}{\om(S)},\quad \z\in\D,
    $$
and
    $$
    \widetilde{M}^d_\om(\mu)(\zeta)=\max\left\{\sup_{\{T(I),\,I\in\Upsilon:\z\in T(I)\}}\frac{\mu{(T(I))}}{\om(T(I))},\frac{\mu{(\D)}}{\om(\D)}\right\},\quad \z\in\D.
    $$
Obviously, $M_\om(\mu)(\zeta)\asymp M^d_\om(\mu)(\zeta)\asymp\widetilde{M}^d_\om(\mu)(\zeta)$ for all $\z\in\D$ if $\om\in\DD$.

Theorem~\ref{Theorem:CarlesonMeasures}(a) is contained in the following result. The estimate \eqref{Eq:CMS-estimate-weak-p>q}, obtained by using the stopping time, together with \cite[Lemma~4.4]{PelRat} shows that  the condition (iv) implies (i) in Theorem~\ref{Thm:Calesonq<p}. However, the proof presented below does not use this fact.

\begin{theorem}\label{Thm:Calesonq<p}
Let $0<q<p<\infty$, $\om\in\DD$ and $\mu$ be a positive Borel measure on~$\D$. Then the following conditions are equivalent:
\begin{enumerate}
\item[\rm(i)] $\mu$ is a $q$-Carleson measure for $A^p_\om$;
\item[\rm(ii)] The function
    $$
    \displaystyle \Psi_\mu(z)=\int_{\D}\left(\frac{1-|\z|}{|1-\overline{\z}z|}\right)^{\lambda}\frac{d\mu(\z)}{\om(T(\z))},\quad z\in\D,
    $$
belongs to $L^{\frac{p}{p-q}}_\om$ for all $\lambda=\lambda(\om)>0$ large enough;
 \item[\rm(iii)] The function
    $$
    B_\mu(z)=\int_{\Gamma(z)}\frac{d\mu(\z)}{\om(T(\z))},\quad
    z\in\D\setminus\{0\},
    $$
belongs to $L^{\frac{p}{p-q}}_\om$;
\item[\rm(iv)] $M_\om(\mu)\in L^{\frac{p}{p-q}}_\om$.
\end{enumerate}
\end{theorem}

We will need the following preliminary result which shows that $\|f\|_{T^p_q(\nu,\om)}\asymp\|C_{q,\nu}(f)\|_{L^p_\om}$ for $0<q<p<\infty$.

\begin{lemma}\label{le:cond i-v} Let $\om\in\DD$ and $\nu$ be a positive Borel measure on $\D$.
\begin{itemize}
\item[\rm(i)] If $0<p,q<\infty$, then
    $
    \|f\|_{T^p_q(\nu,\om)}\lesssim\|C_{q,\nu}(f)\|_{L^p_\om}
    $
for all $\nu$-measurable functions $f$.

\item[\rm(ii)] If $0<q<p<\infty$, then $\|C_{q,\nu}(f)\|_{L^p_\om}\lesssim\|f\|_{T^p_q(\nu,\om)}$ for all $\nu$-measurable functions $f$.
\end{itemize}
\end{lemma}

\begin{proof}
(i)\, Since $\|f\|_{T^p_q(\nu,\om)}=\|A_{sq,\nu}(|f|^\frac1s)\|_{L^{sp}_\om}^s$ and $\|C_{q,\nu}(f)\|_{L^p_\om}=\|C_{sq,\nu}(|f|^\frac1s)\|_{L^{sp}_\om}^s$ for $0<s<\infty$, we may assume that $1<p,q<\infty$.
Then Theorem~\ref{Thm:tent-spaces-duality} and \eqref{Eq:CMS-estimate} yield
    \begin{equation*}\label{j50}
    \begin{split}
    \|f\|_{T^p_q(\nu,\om)}&\lesssim\sup_{\|g\|_{T^{p'}_{q'}(\nu,\om)}\le1}\int_{\D} A_{q',\nu}(g)(\z)C_{q,\nu}(f)(\z)\om(\z)\,dA(\z)\le\|C_{q,\nu}(f)\|_{L^p_\om},
    \end{split}\end{equation*}
and the assertion is proved. 

(ii) Let $a'=(1-2(1-|a|))\frac{a}{|a|}$ if $\frac12<|a|<1$. Then Fubini's theorem yields
    \begin{equation*}
    \begin{split}
    &\frac{1}{\om(T(a))}\int_{T(a)}|f(z)|^q\om(T(z))\,d\nu(z)\\
    &\le\frac{1}{\om(T(a))}\int_{S(a')}\int_{T(a)\cap\Gamma(\z)}|f(z)|^q\,d\nu(z)\om(\z)\,dA(\z)\\
    &\lesssim\frac{1}{\om(S(a'))}\int_{S(a')}A_{q,\nu}^q(f)(\z)\om(\z)\,dA(\z),\quad|a|>\frac12.
    \end{split}
    \end{equation*}
If now $u\in T(a)\subset S(a')$, we obtain
    \begin{equation*}
    \begin{split}
    \frac{1}{\om(T(a))}\int_{T(a)}|f(z)|^q\om(T(z))\,d\nu(z)\lesssim M_\om(A^q_{q,\nu}(f))(u),
    \end{split}
    \end{equation*}
and hence $C^q_{q,\nu}(f)(u)\lesssim M_\om(A^q_{q,\nu}(f))(u)$. This together with Theorem~\ref{co:maxbou} yields
    $$
    \|C_{q,\nu}(f)\|_{L^p_\om}
    \lesssim\|M_\om(A^q_{q,\nu}(f))\|^\frac{1}{q}_{L^{\frac{p}{q}}_\om}
    \lesssim\|A_{q,\nu}(f)\|_{L^p_\om},
    $$
and the assertion follows.
\end{proof}

\medskip

\noindent{\em{Proof of Theorem~\ref{Thm:Calesonq<p}}}. The equivalence (ii)$\Leftrightarrow$(iii) follows by choosing $d\nu(z)=\frac{d\mu(z)}{\om(T(z))}$ in Lemma~\ref{Lemma:cone-integral-nu}. We will show next that (iii)$\Rightarrow$(i)$\Rightarrow$(iv) and (iv)$\Rightarrow$(ii).

(iii)$\Rightarrow$(i). Fubini's theorem, H\"{o}lder's inequality and \cite[Lemma~4.4]{PelRat} yield
    \begin{equation*}
    \begin{split}
    \|f\|_{L^q(\mu)}^q\le\int_{\D}(N(f)(\z))^qB_\mu(\zeta)\om(\z)\,dA(\z)\lesssim\|f\|_{A^p_\om}^q \|B_\mu\|_{L^{\frac{p}{p-q}}_\om}.
    \end{split}
    \end{equation*}

(i)$\Rightarrow$(iv). Let $\{z_{k}\}$ be a separated sequence such that $z_k\ne 0$ for all $k$. By Lemma~\ref{Lemma:operator-tent-bergman}, there exists $\lambda=\lambda(\om)>1$ such that $S_\lambda:T^p_2(\{z_k\},\om)\to A^p_\om$ is bounded. By denoting $\{b_{k}\}=\{f(z_{k})\}$, this together with (i) implies
    \begin{equation}\label{111}
    \begin{split}
    \int_{\D}\left|\sum_{k}b_{k}h_\lambda(z_{k},z)\right|^q\,d\mu(z)
    &=\|S_\lambda(f)\|^q_{L^q(\mu)}
    \lesssim\|S_\lambda(f)\|^q_{A^p_\om}
    \lesssim\|\{b_{k}\}\|^q_{T^p_2(\{z_{k}\},\om)}.
    \end{split}
    \end{equation}
By replacing $b_k$ by $r_k(t)b_k$, where $r_k$ denotes the $k$th Rademacher function, in \eqref{111}, and applying Khinchine's inequality, we deduce
    \begin{equation*}
    \int_{\D}\left(\sum_{k}|b_{k}|^2
    \left|h_\lambda(z_{k},z)\right|^2\right)^{\frac{q}{2}}\,d\mu(z)
    \lesssim\|\{b_{k}\}\|^q_{T^p_2(\{z_{k}\},\om)}.
    \end{equation*}
Since $|h_\lambda(z_{k},z)|\gtrsim\chi_{T(z_{k})}(z)$ for $z\in T(z_{k})$, we deduce
    \begin{equation}\label{7.5}
    \int_{\D}\left(\sum_{k}|b_{k}|^2
    \chi_{T(z_{k})}(z)\right)^{\frac{q}{2}}\,d\mu(z)
    \lesssim\|\{b_{k}\}\|^q_{T^p_2(\{z_{k}\},\om)}.
    \end{equation}
We will now construct $\{b_k\}$ and $\{z_k\}$ appropriately.
For each $k\in\mathbb{Z}$, let $\mathcal{E}_k$ denote the collection of maximal dyadic tents induced by $\Upsilon$ with respect to inclusion such that
$\mu(T)>2^k\om(T)$, and let $\mathcal{E}=\cup_{k}\mathcal{E}_k$. Further, let $E_k=\cup_{\{T\in \mathcal{E}_k\}}T$. Then
$2^k<\widetilde{M}^d_\om(\mu)(z)\le 2^{k+1}$ for $z\in E_k\setminus E_{k+1}$.
Let $\{b_T\}$ be any sequence indexed by $T\in\mathcal{E}$. Assume for a moment that $\mu$ has a compact support. Then $\{b_T\}$ is a finite sequence. Let $G(T)=T-\cup\{T': T'\in\mathcal{E}, T'\subset T\}$, and hence $G(T)=T-\cup\{T': T'\in\mathcal{E}_{k+1}, T'\subset T\}$ for $T\in \mathcal{E}_{k}$. Since $G(T_1)$ and $G(T_2)$ are disjoint for two different tents in $\mathcal{E}$, we have
    \begin{equation}\label{7.51}
    \left(\sum|b_{T}|^2\chi_{T}(z)\right)^{\frac{q}{2}}
    \ge\left(\sum|b_{T}|^2\chi_{G(T)}(z)\right)^{\frac{q}{2}}= \sum|b_{T}|^q \chi_{G(T)}(z).
    \end{equation}
Let us now index the the tents in $\mathcal{E}$ according to which $\mathcal{E}_{k}$ they belong to. Write $\mathcal{E}_{k}=\{T^k_j:j\in\N \}$,
$b_{T^k_j}=b_{k,j}$ and $z_{k,j}$ for the vertex of $T^k_j$. Then \eqref{7.5} and \eqref{7.51} yield
    \begin{equation}
    \begin{split}\label{7.6}
    &\sum_k\sum_j b_{k,j}^q\left(\mu(T^k_j)-\sum_{T^{k+1}_i\subset T^k_j}\mu(T^{k+1}_i)\right)\lesssim\|\{b_{k,j}\}\|^q_{T^p_2(\{z_{k,j}\},\om)}\\
    &=\left(\int_\D \left( \sum_{k,j}|b_{k,j}|^2\chi_{T^k_j}(\z)\right)^{p/2}\om(\z)\, dA(\z) \right)^{q/p}.
    \end{split}
    \end{equation}
Further, write $r=\frac{p}{q}$ and choose $b_{k,j}^q=2^{k(r'-1)}$ for each $j$. Then, by using the inequality $2^k<\widetilde{M}^d_\om(\mu)(z)\le 2^{k+1}$ for $z\in E_k\setminus E_{k+1}$, the left-hand side of \eqref{7.6} can be estimated as
    \begin{equation*}
    \begin{split}
    &\sum_k 2^{k(r'-1)}\mu(E_k)-\sum_k 2^{k(r'-1)}\sum_j\mu(T^k_j\cap E_{k+1})
    =\sum_k \left(2^{k(r'-1)}-2^{(k-1)(r'-1)}\right)\mu(E_k)
    \\ & =\left(1-\frac{1}{2^{r'-1}}\right)\sum_k 2^{k(r'-1)}\sum_{j}\frac{\mu(T^k_j)}{\om(T^k_j)}\om(T^k_j)
    \gtrsim\sum_k 2^{kr'}\om(E_k)\\
    &\gtrsim \sum_{k} \int_{E_k\setminus E_{k+1}} \left(\widetilde{M}^d_\om(\mu)(\z)\right)^{r'}\om(\z)\,dA(\z),
    \end{split}
    \end{equation*}
while the right-hand side of \eqref{7.6}, with the notation $\eta=2^{(r'-1)\frac{2}{q}}$, is comparable to
    \begin{equation*}
    \begin{split}
    & \left(\int_\D  \left( \sum_{k}2^{(k-1)(r'-1))\frac{2}{q}}\chi_{E_k}(\z)\right)^{p/2}\om(\z)\, dA(\z) \right)^{1/r}
    \\ & \asymp\left(\int_\D  \left( \sum_{k}(\eta^k-\eta^{k-1})\chi_{E_k}(\z)\right)^{p/2}\om(\z)\, dA(\z) \right)^{1/r}
 \\ &  = \left(\int_\D \sum_{k} \eta^{\frac{kp}{2}}\chi_{E_k\setminus E_{k+1}}(\z)\om(\z)\, dA(\z) \right)^{1/r}
 \\ & \asymp \left(\sum_{k} \int_{E_k\setminus E_{k+1}}\left(\widetilde{M}^d_\om(\mu)(\z)\right)^{r'}\om(\z)\,dA(\z)\right)^{1/r}.
    \end{split}
    \end{equation*}
Combining these two estimates we find a constant $C>0$, independent of $\mu$, such that $\|\widetilde{M}^d_\om(\mu)\|_{L^{r'}_\om}\le C$,
and so $\widetilde{M}^d_\om(\mu)\in L^{\frac{p}{p-q}}_\om$ for $\mu$ with compact support. The result for arbitrary $\mu$
follows from this by a standard argument based on the monotone convergence theorem.

Finally, let us observe that an argument similar to that of Lemma~\ref{Lemma:cone-integral-nu} (the case $0<p<1$) yields
(iv)$\Rightarrow$(ii). To see this, let $g(u)=\frac{\mu(\Delta(u,r))}{\om(T(u))}$. Then (iv) is equivalent to $C_{1,h}(g)\in L^{\frac{p}{p-q}}_\om$ and (ii) is equivalent to $g\in T^{\frac{p}{p-q}}_1(h,\om)$. Since $\|g\|_{T^{\frac{p}{p-q}}_1(h,\om)}=\|g^{1/2}\|^2_{T^{\frac{2p}{p-q}}_2(h,\om)}$ and
$\|C_{1,h}(g)\|_{L^\frac{p}{p-q}_\om}=\|C_{2,h}(g^{1/2})\|^2_{L^\frac{2p}{p-q}_\om}$, the implication (iv)$\Rightarrow$(ii) follows by  Lemma~\ref{le:cond i-v}.\hfill $\Box$

\subsection{Case $0<p\le q<\infty$}

We begin with establishing necessary conditions for $D^{(n)}:A^p_\om\to L^q(\mu)$ to be bounded when $n\in\N\cup\{0\}$.

\begin{lemma}\label{Lemma:necessary-conditions:q>p}
Let $0<p,q<\infty$, $\om\in\DD$ and $n\in\N\cup\{0\}$, and let $\mu$ be a positive Borel measure on $\D$. If
$D^{(n)}:A^p_\om\to L^q(\mu)$ is bounded, then the function
    $$
    z\mapsto\frac{\mu(S(z))}{\om(S(z))^\frac{q}{p}(1-|z|)^{nq}}
    $$
belongs to $L^\infty$. In the definition above, $\mu(S(z))$ can be replaced by $\mu(\Delta(z,r))$ for any fixed $r\in(0,1)$.
\end{lemma}

\begin{proof}
For $a\in\D$, define
    $$
    f_{a,p}(z)=\left(\frac{1-|a|}{1-\overline{a}z}\right)^\frac{\lambda+1}{p}\om(S(a))^{-\frac1p},\quad z\in\D,
    $$
where $\lambda=\lambda(\om)>0$ is as in Lemma~\ref{Lemma:replacement-Lemma1.1-2.}. Then
    \begin{equation*}
    \begin{split}
    1&\asymp\|f_{a,p}\|_{A^p_\om}^q\gtrsim\|f_{a,p}^{(n)}\|_{L^q(\mu)}^q
    \gtrsim\frac{\mu(S(a))}{\om(S(a))^\frac{q}{p}(1-|a|)^{nq}},\quad a\in\D.
    \end{split}
    \end{equation*}
Since the same reasoning works when $\mu(S(a))$ is replaced by $\mu(\Delta(a,r))$ for any fixed $r\in(0,1)$, the lemma is proved.
\end{proof}

We remind the reader that \eqref{Eq:CMS-estimate-weak} together with \cite[Lemma~4.4]{PelRat} shows that $\mu$ is a $p$-Carleson measure for $A^p_\om$ if $M_\om(\mu)\in L^\infty$. The next theorem concerns the case $q\ge p$ and, in particular, gives an alternative way to deduce this implication.

\begin{theorem}\label{Thm:Carleson-Withney}
Let $0<p\le q<\infty$, $\om\in\DD$, and let $\mu$ be a
positive Borel measure on $\D$. For $f:\D\to\C$ and $0<\lambda<\infty$, denote $O_\lambda(f)=\{z\in\D:N(f)(z)>\lambda\}\subset\D$. Then
    $$
    \mu(O_\lambda(f))\lesssim\left(\int_{O_\lambda(f)}\left(M_{\om,q/p}(\mu)(\z)\right)^\frac{p}{q}\om(\z)\,dA(\z)\right)^\frac{q}{p},\quad 0<\lambda<\infty,
    $$
for all $f:\D\to\C$, and hence
    \begin{equation}\label{Eq:Carleson-n=0}
    \|N(f)\|_{L^q(\mu)}\lesssim\|N(f)M^\frac{1}{q}_{\om,q/p}(\mu)\|_{L^p_\om}.
    \end{equation}
\end{theorem}

\begin{proof}
Consider the open sets $O_\lambda(f)$ and $\widehat{O}_\lambda(f)=\left\{z/|z|:z\in O_\lambda(f)\right\}\subset\T$ for $0<\lambda<\infty$.
Let $\cup_j I^\lambda_j=\widehat{O}_\lambda(f)$ be the Whitney covering of $\widehat{O}_\lambda(f)$ in the same extended sense as in the proof of Theorem~\ref{Proposition:atoms}. For each $\lambda$ and $j$, let $J^\lambda_j\subset\T$ be an arc with the same center as $I^k_j$, but $c$ times its length. Now, by choosing $c$ sufficiently large, we have $O_\lambda(f)\subset\cup_j T(J^\lambda_j)$.
Moreover, by the hypothesis and Lemma~\ref{Lemma:replacement-Lemma1.1-2.}, $\mu(T(J^\lambda_j))\lesssim \left(\om(T(I^\lambda_j))\right)^\frac{q}{p}\inf_{\z\in T(J^\lambda_j)}M_{\om,q/p}(\mu)(\z)$, and hence
   \begin{equation*}
    \begin{split}
    \mu(O_\lambda(f)) &\le \sum_j\mu(T(J^\lambda_j))\lesssim
    \sum_j\left(\om(T(I^\lambda_j))\right)^\frac{q}{p}\inf_{\z\in T(J^\lambda_j)}M_{\om,q/p}(\mu)(\z)\\
    &\le \left(\sum_j\om(T(I^\lambda_j))\right)^\frac{q}{p}\inf_{\z\in T(J^\lambda_j)}M_{\om,q/p}(\mu)(\z)\\
    &\lesssim \left(\int_{O_\lambda(f)}\left(M_{\om,q/p}(\mu)(\z)\right)^\frac{p}{q}\om(\z)\,dA(\z)\right)^\frac{q}{p}.
    \end{split}
    \end{equation*}
An integration with respect to $\lambda$ together with Minkowski's inequality in continuous form (Fubini's theorem in the case $q=p$) yields \eqref{Eq:Carleson-n=0}.
\end{proof}

We next establish a sufficient condition for $D^{(n)}:A^p_\om\to L^q(\mu)$ to be bounded when~$p<q$.

\begin{proposition}\label{Prop:p<q-sufficient}
Let $0<p<q<\infty$, $\om\in\DD$ and $n\in\N\cup\{0\}$, and let $\mu$ be a positive Borel measure on $\D$. If the function
    $$
    z\mapsto\frac{\mu(\Delta(z,r))}{\om(S(z))^\frac{q}{p}(1-|z|)^{nq}}
    $$
belongs to $L^\infty$ for a fixed $r\in(0,1)$, then $D^{(n)}:A^p_\om\to L^q(\mu)$ is bounded.
\end{proposition}

\begin{proof}
The argument follows the proof of \cite[Theorem~3.1]{Luecking1985}. Let $r\in(0,1)$ be fixed, choose $s\in(p,q)$ and denote
    $$
    d\mu^\star(\z)=\frac{(\om(S(\z)))^\frac{s}{p}}{(1-|\z|)^{2}}\,dA(\z),\quad \z\in\D.
    $$
Recall the known estimate
    \begin{equation}\label{Eq:suharmonic-n-derivatives}
    |f^{(n)}(z)|^s\lesssim\frac{1}{(1-|z|)^{2+ns}}\int_{\Delta(z,r)}|f(\z)|^s\,dA(\z),\quad z\in\D,
    \end{equation}
see, for example, \cite[Lemma~2.1]{Luecking1985} for details. This estimate, Minkowski's inequality in continuous form and the assumption give
    \begin{equation}\label{Eq:Carleson-differentiation-form}
    \begin{split}
    \|f^{(n)}\|_{L^q(\mu)}^q&\lesssim\int_\D\left(\frac{1}{(1-|z|)^{2+ns}}\int_{\Delta(z,r)}|f(\z)|^s\,dA(\z)\right)^\frac{q}{s}\,d\mu(z)\\
    &\lesssim\left(\int_\D|f(\z)|^s\frac{(\mu(\Delta(\z,r)))^\frac{s}{q}}{(1-|\z|)^{ns}}\,dh(\z)\right)^\frac{q}{s}\\
    &\lesssim\left(\int_\D|f(\z)|^s(\om(S(\z)))^\frac{s}{p}\,dh(\z)\right)^\frac{q}{s}.
    \end{split}
    \end{equation}
Standard arguments show that 
$\mu^\star(S(a))\lesssim(\om(S(a)))^\frac{s}{p}$ for $a\in\D$. Hence, by 
Theorem~\ref{Thm:Carleson-Withney} and \cite[Lemma~4.4]{PelRat}
we obtain
    $
    \|f^{(n)}\|_{L^q(\mu)}\lesssim \|f\|_{L^s(\mu^\star)}^\frac{s}{p}
    \lesssim\|f\|_{A^p_\om},
    $
and the assertion is proved.
\end{proof}

By combining the results above we will characterize $q$-Carleson measures for~$A^p_\om$. Parts (b) and (c) of Theorem~\ref{Theorem:CarlesonMeasures} are contained in the following result.

\begin{theorem}\label{Theorem:CarlesonMeasures-cases2-3}
Let $0<p\le q<\infty$, $\om\in\DD$ and $\mu$ be a positive Borel measure on~$\D$. Further, for $0<\lambda<\infty$ and $f:\D\to\C$, denote $O_\lambda(f)=\{z\in\D:N(f)(z)>\lambda\}$.
\begin{itemize}
\item[\rm(a)] The following conditions are equivalent:
\begin{enumerate}
\item[\rm(i)] $\mu$ is a $p$-Carleson measure for $A^p_\om$;
\item[\rm(ii)] $\displaystyle\sup_{f,\la}\frac{\mu(O_\lambda(f))}{\om(O_\lambda(f))}<\infty$;
\item[\rm(iii)] $M_{\om}(\mu)\in L^\infty$;
\end{enumerate}
\item[\rm(b)] If $q>p$, then the following conditions are equivalent:
\begin{enumerate}
\item[\rm(i)] $\mu$ is a $q$-Carleson measure for $A^p_\om$;
\item[\rm(ii)] $\displaystyle\sup_{f,\la}\frac{\mu(O_\lambda(f))}{(\om(O_\lambda(f)))^\frac{q}{p}}<\infty$;
\item[\rm(iii)] $M_{\om,q/p}(\mu)\in L^\infty$;
\item[\rm(iv)]  $\displaystyle z\mapsto\frac{\mu\left(\Delta(z,r)\right)}{(\om\left(S(z)
    \right))^\frac{q}{p}}$ belongs to $L^\infty$ for any fixed $r\in(0,1)$.
\end{enumerate}
\end{itemize}
\end{theorem}

\noindent{\em{Proof of Theorem~\ref{Theorem:CarlesonMeasures-cases2-3}}}. The fact that (iii) and (iv) are necessary conditions for $\mu$ to be Carleson follows by Lemma~\ref{Lemma:necessary-conditions:q>p}. The implications (iii)$\Rightarrow$(ii)$\Rightarrow$(i) for $q\ge p$ follow by Theorem~\ref{Thm:Carleson-Withney} and \cite[Lemma~4.4]{PelRat}. Proposition~\ref{Prop:p<q-sufficient} implies (iv)$\Rightarrow$(i).
\hfill $\Box$
\medskip

By choosing $d\mu(z)=\om(S(z))(1-|z|)^{-2}\,dA(z)$ (this measure is even  not necessarily finite), we see that $\mu(\Delta(a,r))\lesssim\om(S(a))$, but $\mu(S(a))/\om(S(a))\to\infty$, as $|a|\to1^-$, if $\frac{\widehat{\om}(r)}{\om(r)(1-r)}\to\infty$, as $r\to1^-$. See \cite[Chapter~1]{PelRat} for more information on the Bergman spaces induced by the weights satisfying this limit condition. Therefore the condition (iv) in Theorem~\ref{Theorem:CarlesonMeasures-cases2-3}(b) does not characterize $p$-Carleson measures for $A^p_\om$.

\section{Differentiation operators from $A^p_\om$ to $L^q(\mu)$}

The following result contains Theorem~\ref{Theorem:DifferentiationOperator}(b).

\begin{theorem}
Let either $0<p<q<\infty$ or $2\le p=q<\infty$, $\om\in\DD$ and $n\in\N$, and let $\mu$ be a positive Borel measure on $\D$. Then the following conditions are equivalent:
\begin{enumerate}
\item[\rm(i)] $D^{(n)}:A^p_\om\to L^q(\mu)$ is bounded;
\item[\rm(ii)] $\displaystyle z\mapsto\frac{\mu\left(S(z)\right)}{\om\left(S(z)\right)^\frac{q}{p}(1-|z|)^{nq}}$ belongs to $L^\infty$;
\item[\rm(iii)] $\displaystyle z\mapsto\frac{\mu\left(\Delta(z,r)\right)}{\om\left(S(z)\right)^\frac{q}{p}(1-|z|)^{nq}}$ belongs to $L^\infty$ for any fixed $r\in(0,1)$.
\end{enumerate}
\end{theorem}

\begin{proof}
By Lemma~\ref{Lemma:necessary-conditions:q>p} and Proposition~\ref{Prop:p<q-sufficient}, it suffices to show that (iii)$\Rightarrow$(i) for $q=p\ge2$. We will prove this implication under the weaker assumption $q\ge\max\{2,p\}$. Let $r\in(0,1)$ be fixed, and let first $n=1$. The standard Paley-Littlewood formula and the fact that the Laplacian $\triangle|f|^q$ is subharmonic when $q\ge2$ give
    \begin{equation*}
    \begin{split}
    |f'(0)|^q\le\frac{q^2}{4}\int_{\D}\triangle|f|^q(z)\log\frac{1}{|z|}\,dA(z)\asymp\int_{\D}\triangle|f|^q(z)(1-|z|)\,dA(z).
    \end{split}
    \end{equation*}
An application of this inequality to the function $f(rz)$ gives
    $$
    |f'(0)|^q\lesssim\int_{\De(0,r)}\triangle|f|^q(z)\left(1-\frac{|z|}{r}\right)\,dA(z).
    $$
Replace now $f$ by $f\circ\vp_z$ to obtain
    \begin{equation}\label{Eq:Laplacian-first-derivative}
    \begin{split}
    |f'(z)|^q(1-|z|^2)^q&\lesssim\int_{\Delta(z,r)}\triangle|f|^q(\z)\left(1-\frac{|\vp_z(\z)|}{r}\right)\,dA(\z)\\
    &\le\int_{\Delta(z,r)}\triangle|f|^q(\z)\,dA(\z).
    \end{split}
    \end{equation}
This estimate, Fubini's theorem, the assumption (iii), the standard estimate $|f(z)|\lesssim\|f\|_{A^p_\om}\om(S(z))^{-p^{-1}}$
and \cite[Theorem~4.2]{PelRat}, together with Lemma~\ref{Lemma:replacement-Lemma1.1-2.}, yield
    \begin{equation}\label{iiiii}
    \begin{split}
    \|f'\|_{L^q(\mu)}^q
    &\lesssim\int_\D\frac{\triangle|f|^q(\z)}{(1-|\z|)^q}\mu(\Delta(\z,r))\,dA(\z)
    \lesssim\int_\D\triangle|f|^q(\z)\om(S(\z))^\frac{q}{p}\,dA(\z)\\
    &\lesssim\|f\|_{A^p_\om}^{q-p}\int_\D\triangle|f|^p(\z)\om(S(\z))\,dA(\z)\lesssim\|f\|_{A^p_\om}^q,
    \end{split}
    \end{equation}
and thus the case $n=1$ is proved.

To prove the general case, apply \eqref{Eq:suharmonic-n-derivatives}, \eqref{Eq:Laplacian-first-derivative}, Fubini's theorem and \eqref{iiiii}, to deduce
    \begin{equation*}
    \begin{split}
    \|f^{(n)}\|_{L^q(\mu)}^q
    &\lesssim\int_\D\frac{1}{(1-|z|)^{2+(n-1)q}}\left(\int_{\Delta(z,\r)}|f'(\z)|^q\,dA(\z)\right)\,d\mu(z)\\
    &\lesssim\int_\D\frac{1}{(1-|z|)^{2+nq}}\left(\int_{\Delta(z,\r)}\left(\int_{\Delta(\z,s)}\triangle|f|^q(u)\,dA(u)\right)\,dA(\z)\right)\,d\mu(z)\\
    &\lesssim\int_\D\frac{1}{(1-|z|)^{nq}}\left(\int_{\Delta(z,r)}\triangle|f|^q(u)\,dA(u)\right)\,d\mu(z)\\
    &\asymp\int_\D\frac{\triangle|f|^q(u)}{(1-|u|)^{nq}}\mu(\Delta(u,r))\,dA(u)\\
    &\lesssim\int_\D\triangle|f|^q(u)\om(S(u))^\frac{q}{p}\,dA(u)\lesssim\|f\|_{A^p_\om}^q,
    \end{split}
    \end{equation*}
where $\r,s\in(0,1)$ are chosen sufficiently small depending only on $r$.
\end{proof}

\begin{theorem}
Let $0<q\le p<\infty$, $\om\in\DD$ and $n\in\N$, and let $\mu$ be a positive Borel measure on $\D$. Then $D^{(n)}:A^p_\om\to L^q(\mu)$ is bounded if and only if, for any fixed $r\in(0,1)$, the function
    $$
    \Phi_\mu(z)=\frac{\mu(\Delta(z,r))}{\om(S(z))(1-|z|)^{qn}}
    $$
belongs to
    \begin{enumerate}
    \item[\rm(i)] $T^\frac{p}{p-q}_{\frac{2}{2-q}}\left(h,\om\right)$, if $q<\min\{2,p\}$;
    \item[\rm(ii)] $T^\infty_{\frac{2}{2-p}}\left(h,\om\right)$, if $q=p<2$;
    \item[\rm(iii)] $T^{\frac{p}{p-q}}_\infty\left(h,\om\right)$, if $2\le q<p$.
    \end{enumerate}
\end{theorem}

\begin{proof} We begin with showing that conditions (i)-(iii) are sufficient.

Let $r\in(0,1)$ be fixed and denote $F=|f'(w)|^q(1-|w|)^q $. Then \eqref{Eq:suharmonic-n-derivatives} and Fubini's theorem give
    \begin{equation}\label{eq:sufdual}
    \begin{split}
    \|f^{(n)}\|_{L^q(\mu)}^q&\lesssim\int_\D |f'(w)|^q\frac{\mu(\Delta(w,r))}{\om(S(w))(1-|w|)^{2+(n-1)q}}\om(T(w))\,dA(w)\\
    &= \langle F,\Phi_\mu\rangle_{T^2_2(h,\om)}.
    \end{split}
    \end{equation}

Under the hypotheses of (i) or (ii), $f\in A^p_\om$ if and only if $F\in T^{p/q}_{2/q}(h,\om)$ with $\|F\|_{T^{p/q}_{2/q}(h,\om)}\lesssim\|f\|^q_{A^p_\om}$ by \cite[Theorem~4.2]{PelRat}. Since $\Phi_\mu\in T^{\left(\frac{p}{q}\right)'}_{\left(\frac{2}{q}\right)'}(h,\om)$ by the assumption, and $T^{\left(\frac{p}{q}\right)'}_{\left(\frac{2}{p}\right)'}(h,\om)\simeq\left(T^{p/q}_{2/q}(h,\om)\right)^\star$ by Theorem~\ref{Thm:tent-spaces-duality}, it follows that $D^{(n)}:A^p_\om\to L^q(\mu)$ is bounded.

If $2\le q<p<\infty$, then \eqref{Eq:Laplacian-first-derivative} and Fubini's theorem give
   \begin{equation*}
    \begin{split}
    \|F\|^{\frac{q}{p}}_{T^{\frac{p}{q}}_{1}(h,\om)}
    &\le\int_\D\left(\int_{\Gamma(\z)}\int_{\Delta(w,r)}\triangle|f|^q(z)\,dA(z)\,dh(w)\right)^{\frac{p}{q}}\om(\z)\,dA(\z)\\
    &\le \int_\D\left(\int_{\Gamma'(\z)}\triangle|f|^q(z)\,dA(z)\right)^{\frac{p}{q}} \om(\z)\,dA(\z),
    \end{split}
    \end{equation*}
where $\Gamma'(\z)=\{z:\Gamma(\z)\cap\Delta(z,r)\ne\emptyset\}$. Arguing as in \eqref{eq:j3}, and using a result by Calder\'on~\cite[Theorem~1.3]{Pavlovic2013}, we get $F\in T^{\frac{p}{q}}_{1}(h,\om)$ with $\|F\|_{T^{\frac{p}{q}}_{1}(h,\om)}\lesssim\|f\|^q_{A^p_\om}$.
Since $\Phi_\mu\in T^{\left(\frac{p}{q}\right)'}_{\infty}(h,\om)$ by the assumption, and $T^{\left(\frac{p}{q}\right)'}_{\infty}(h,\om)\simeq\left(T^{p/q}_{1}(h,\om)\right)^\star$ by Theorem~\ref{Thm:tent-spaces-duality}, it follows
by \eqref{eq:sufdual} that $D^{(n)}:A^p_\om\to L^q(\mu)$ is bounded.

To see the necessity of conditions (i)-(iii), let $\{z_k\}$ be a $\d$-lattice such that $z_k\ne0$ for all $k$, and consider the operator
    $$
    S_\lambda(f)(z)=\sum_k f(z_k)\left(\frac{1-|z_k|}{1-\overline{z}_kz}\right)^\lambda,\quad z\in\D,
    $$
defined in Lemma~\ref{Lemma:operator-tent-bergman}. Then the assumption and Lemma~\ref{Lemma:operator-tent-bergman} yield
    \begin{equation*}
    \int_\D\left|\sum_kf(z_k)\frac{(1-|z_k|)^{\lambda}}{(1-\overline{z}_kz)^{\lambda+n}}\right|^q\,d\mu(z)
    \lesssim\|S_\lambda(f)\|_{A^p_\om}^q\lesssim\|f\|_{T^p_2(\{z_k\},\om)}^q.
    \end{equation*}
Replace now $f(z_k)$ by $f(z_k)r_k(t)$ and integrate with respect to $t$ to obtain
    $$
    \int_0^1\int_\D\left|\sum_kf(z_k)\frac{(1-|z_k|)^{\lambda}}{(1-\overline{z}_kz)^{\lambda+n}}r_k(t)\right|^qd\mu(z)\,dt\lesssim\|f\|_{T^p_2(\{z_k\},\om)}^q,
    $$
from which Fubini's theorem and a standard application of Khinchine's inequality yield
    \begin{equation*}
    I=\int_\D\left(\sum_k|f(z_k)|^2\frac{(1-|z_k|)^{2\lambda}}{|1-\overline{z}_kz|^{2\lambda+2n}}\right)^\frac{q}{2}\,d\mu(z)\lesssim\|f\|_{T^p_2(\{z_k\},\om)}^q.
    \end{equation*}
Now, for any fixed $r\in(0,1)$,
    \begin{equation*}
    \begin{split}
    I\gtrsim\sum_j\frac{|f(z_j)|^q}{(1-|z_j|)^{qn}}\mu(\Delta(z_j,r)),
    \end{split}
    \end{equation*}
and hence
    \begin{equation}\label{pippeli}
    \begin{split}
    &\sum_j\frac{|f(z_j)|^q}{(1-|z_j|)^{qn}}\mu(\Delta(z_j,r))\lesssim\|f\|_{T^p_2(\{z_k\},\om)}^q\\
    &=\left(\int_\D\left(\left(\sum_{z_k\in\Gamma(\z)}\left(|f(z_k)|^q\right)^\frac{2}{q}\right)^\frac{q}{2}\right)^\frac{p}{q}\om(\z)\,dA(\z)\right)^\frac{q}{p}.
    \end{split}
    \end{equation}

We now treat different cases separately.

(i) If $q<\min\{2,p\}$, then $s=\frac{p}{q}>1$ and $r=\frac{2}{q}>1$, and hence $(T^s_r(\{z_k\},\om))^\star\simeq T^{s'}_{r'}(\{z_k\},\om))$ by Theorem~\ref{Thm:tent-spaces-duality}. Therefore \eqref{pippeli} yields
    $$
    \int_\D\left(\sum_{z_k\in\Gamma(\z)}\left(\frac{\mu(\Delta(z_k,r))}{\om(T(z_k))(1-|z_k|)^{qn}}\right)^\frac{2}{2-q}\right)
    ^{\frac{2-q}{2}\frac{p}{p-q}}\om(\z)\,dA(\z)<\infty.
    $$
Take $t\in(0,\d^{-1})$ sufficiently large such that $\Gamma(\z)\subset\cup_{z_k\in\Gamma(\z)}\Delta(z_k,t\d)$ for all $\z\in\D$, and then $r\in(0,1)$ sufficiently large such that $\Delta(z,\r)\subset\Delta(z_k,r)$ for all $z\in\Delta(z_k,t\d)$ and all $k$. Then Lemma~\ref{Lemma:replacement-Lemma1.1-2.} yields
    \begin{equation*}
    \begin{split}
    &\int_{\Gamma(\z)}\left(\frac{\mu(\Delta(z,\r))}{\om(T(z))(1-|z|)^{qn}}\right)^{\frac{2}{2-q}}\,dh(z)\\
    &\le\sum_{z_k\in\Gamma(\z)}\int_{\Delta(z_k,t\d)}\left(\frac{\mu(\Delta(z,\r))}{\om(T(z))(1-|z|)^{qn}}\right)^{\frac{2}{2-q}}\,dh(z)\\
    &\asymp\sum_{z_k\in\Gamma(\z)}\left(\frac{1}{\om(T(z_k))(1-|z_k|)^{qn}}\right)^{\frac{2}{2-q}}\int_{\Delta(z_k,t\d)}
    \left(\mu(\Delta(z,\r))\right)^{\frac{2}{2-q}}\,dh(z)\\
    &\lesssim\sum_{z_k\in\Gamma(\z)}\left(\frac{\mu(\Delta(z_k,r))}{\om(T(z_k))(1-|z_k|)^{qn}}\right)^{\frac{2}{2-q}},
    \end{split}
    \end{equation*}
and the assertion follows.

(ii) If $q=p<2$, then $s=\frac{p}{q}=1$ and $r=\frac{2}{q}>1$, and hence $(T^1_r(\{z_k\},\om))^\star\simeq T^\infty_{r'}(\{z_k\},\om))$ by Theorem~\ref{Thm:tent-spaces-duality}. Therefore \eqref{pippeli} yields
    $$
    \sup_{a\in\D}\frac1{\om(T(a))}\sum_{z_k\in T(a)}\left(\frac{\mu(\Delta(z_k,r))}{\om(T(z_k))(1-|z_k|)^{pn}}\right)^\frac{2}{2-p}\om(T(z_k))<\infty
    $$
for all $r\in(0,1)$. The assertion follows.

(iii) If $2<q<p$, then $s=\frac{p}{q}>1$ and $r=\frac{2}{q}<1$, and hence Proposition~\ref{Proposition:duality-sequence-tent} yields
    $$
    \int_\D\left(\sup_{z_k\in\Gamma(\z)}\frac{\mu(\Delta(z_k,r))}{\om(T(z_k))(1-|z_k|)^{qn}}\right)^\frac{p}{p-q}\om(\z)\,dA(\z)<\infty,\quad 0<r<1,
    $$
from which the assertion follows as in the case (i). The case $q=2$ is proved similarly by using Theorem~\ref{Thm:tent-spaces-duality} instead of Proposition~\ref{Proposition:duality-sequence-tent}.
\end{proof}


\end{document}